\documentclass[11pt]{article}
\usepackage[english]{babel}

\usepackage[utf8]{inputenc}

\usepackage{amsmath}
\usepackage{amsfonts}
\usepackage{amssymb}
\usepackage{amsthm} 
\usepackage{stmaryrd}
\usepackage{dsfont}

\usepackage{tikz}
\usepackage{graphicx}													
\usepackage{tabularx}													\usepackage{multirow}													
\usepackage{booktabs}
\usepackage[labelformat=brace]{subcaption}

\usepackage[linesnumbered,vlined,ruled]{algorithm2e}

\theoremstyle{plain}
\newtheorem{theorem}{Theorem}[section]
\newtheorem{lemma}[theorem]{Lemma}
\newtheorem{corollary}[theorem]{Corollary}
\newtheorem{prop}[theorem]{Proposition}

\theoremstyle{definition}
\newtheorem{definition}{Definition}[section]
\theoremstyle{remark}

\newtheorem{remark}[theorem]{Remark}
\newtheorem{example}[theorem]{Example}

\newtheorem*{continuancex}{Example \continuanceref (continued)}
\newenvironment{continueexample}[1]
  {\newcommand\continuanceref{\ref{#1}}\continuancex}
  {\endcontinuancex}
  
  \numberwithin{equation}{section}

\usepackage{url}

\usepackage{hyperref}

\usepackage{graphicx}
\usepackage{xcolor}

\usepackage{array}
\usepackage[titletoc,title]{appendix}

\newcommand{\OT}{\text{OT}}
\newcommand{\oOT}{\overline{\OT}}
\newcommand{\DOT}{\text{DOT}}

\newcommand{\RR}{\mathbb{R}}
\newcommand{\Pcal}{\mathcal{P}}
\newcommand{\Xcal}{\mathcal{X}}
\newcommand{\Ycal}{\mathcal{Y}}
\newcommand{\eps}{\varepsilon}
\renewcommand{\epsilon}{\varepsilon}

\usepackage{titlesec}
\titleformat{\section}[runin]{\normalfont\bfseries}{\thesection .}{0.5em}{}
\titleformat{\subsection}[runin]{\normalfont\bfseries}{\thesubsection .}{0.5em}{}

\usepackage{changes}

\usepackage{todonotes}

\newcommand{\de}{\mathrm{d}}

\DeclareMathOperator{\argmin}{argmin}

\begin{document}

\title{Characterization of transport optimizers via graphs and\\ applications to Stackelberg-Cournot-Nash equilibria
}
\author{Beatrice Acciaio\thanks{ETH Zurich, Department of Mathematics, \emph{beatrice.acciaio@math.ethz.ch}}\;\; and\;\, Berenice Anne Neumann\thanks{Trier University, Department IV - Mathematics, \emph{neumannb@uni-trier.de}}}
\maketitle
	
\allowdisplaybreaks

\abstract{
We introduce graphs associated to transport problems between discrete marginals, that allow to characterize the set of all optimizers given one primal optimizer. 
In particular, we establish that connectivity of those graphs is a necessary and sufficient condition for uniqueness of the dual optimizers.
Moreover, we provide an algorithm that can efficiently compute the dual optimizer that is the limit, as the regularization parameter goes to zero, of the dual entropic optimizers. 
Our results find an application in a
Stackelberg-Cournot-Nash game, for which we obtain existence and characterization of the equilibria.\\[0.4cm]
}

\section{Introduction.}

Starting with the seminal works by Monge \cite{monge1781memoire} and Kantorovich \cite{kantorovich1942translocation}, optimal transport theory became a vibrant field of research with many applications in the most various fields, including economics, finance and machine learning, see e.g. \cite{galichon2018optimal,acciaio2016model,acciaio2020causal,backhoff2020adapted,acciaio2020cot,acciaio2021cournot}. 
Optimal transport is concerned with the question of how a given probability distribution can be coupled with another given probability distribution in a cost efficient way. 
Under weak conditions, one can show equivalence of this (primal) problem and a dual one. In the case of continuous measures on Euclidean spaces, in many situations of interest the dual optimizer is unique (up to translation), see  \cite[Appendix B]{BerntonUniqueness}, \cite[Cor. 2.7]{delBarrioUniqueness}, \cite[Prop. 7.18]{santambrogio2015optimal} and \cite{StaudtNonUniqueness}. However, in the case where the marginals have finite support, there are several natural examples illustrating how uniqueness can easily fail. To the best of our knowledge, only a sufficient criterion for uniqueness has been described in \cite{StaudtNonUniqueness}. Since non-uniqueness happens regularly, it is of interest to determine the set of all dual optimizers. However, the classical algorithms for the optimal transport problem (see \cite{ComputationalOT} for an overview) output only one (nearly optimal) solution.

In the present work, we use tools from graph theory in order to characterize the set of all optimizers in optimal transport problems with finitely supported marginals. Other authors have been working at the intersection of the two theories, for example studying optimal transport on graphs as in \cite{leonard2016lazy}, or in order to analyze a multi-marginal optimal transport problem as in \cite{pass2021monge}. Our approach goes in a different direction. The starting point consists in associating to each optimal transport problem (in a discrete setting) a family of graphs $G^\gamma$, where $\gamma$ varies over all (primal) optimizers, and $G^\gamma$ describes its support. 
What is crucial for our analysis is the connectivity of the graph $G$ obtained as union of all graphs $G^\gamma$. This allows us to characterize uniqueness (up to translation) of the optimizer to the dual transport problem, in the sense that uniqueness holds if and only if $G$ is connected. On the other hand, when connectivity fails, we can describe in a simple way the set of all dual optimizers by starting with one primal optimizer $\gamma$ of the original transport problem. This is achieved by decomposing the optimal transport problem in subproblems (corresponding to the connected components of the graph $G^\gamma$). For every subproblem, the corresponding graph is connected and a unique dual optimizer is obtained. From the  dual optimizers for the subproblems, we can then determine the set of all dual optimizers for the original problem.

A second contribution of this paper regards the approximation by entropic regularization.
Notably, by adding an entropic penalization term to the transport problem, one obtains a strictly convex problem such that both primal and dual problem admit a unique solution. The entropic transport problem gained popularity because of its computational tractability, e.g. via the Sinkhorn algorithm of Cuturi \cite{cuturi2013sinkhorn}, and since it provides an approximation of the original transport problem, in the sense that the optimal costs converge and that the optimizers of the regularized problems converge to an optimizer of the original problem, see \cite{NutzWieselConvergenceDual,ComputationalOT}. 
In a discrete setting, Cominetti and San Mart\'in \cite{CominettiSanMartin} show that the limit of the dual optimizers of the regularized problems is a specific dual optimizer for the original problem and in \cite{weed2018entropic_convergence} it is shown that the convergence is exponentially fast. The limit optimizer is called centroid because it is characterized geometrically as a particular ``central'' point of the convex set of all optimizers. However, the computation is complex, as it is necessary to solve several nested convex optimization problems.
Our contribution in this direction is twofold. First of all, we provide a simple algorithm to compute the centroid. Further, by leveraging on the previously exposed results, we can describe the set of all dual optimizers of the original problem, starting from the unique primal optimizers of the entropic ones. Given the efficiency in computing the latter, this provides a tractable way to find all dual transport optimizers.

Our last main contribution concerns an application to static games with a continuum of agents as introduced by Aumann \cite{AumannContinuum1964,aumann1966continuum}. In these games, agents with different types choose among a set of actions to minimize their costs which depend on their own type and action as well as on the actions of all other agents (mean-field interaction).
Existence of (Cournot-Nash) equilibria has then been established by Schmeidler~\cite{SchmeidlerStaticContinuum}, Mas-Colell~\cite{MasColellStaticContinuum} and Khan~\cite{khan1989cournot}. However, relying on classical game theory, no further results regarding for example  uniqueness or characterization of equilibria had been obtained until the work of Blanchet and Carlier~\cite{BlanchetCN}. The authors there introduced a class of games with separable cost functions, where equilibria can be characterized by minimizing a cost function that includes an optimal transport problem, and proposed a uniqueness criterion.
In this work we consider a Stackelberg version of this game, where in addition a principal is participating to the game, setting up some cost to be paid by the agents according to their action, and at the same time facing a cost that depends both on this and on the distribution of actions of the agents. Relying on the connection of these games with optimal transport established in \cite{BlanchetCN,BlanchetCNFinite}, we find conditions to ensure existence of equilibria. Interestingly, the optimal choice of costs for the principal corresponds to finding a dual optimizer to an optimal transport problem. We can therefore apply the results illustrated above to describe the optimal strategies of the principal. We conclude by investigating whether entropic regularization gives nearly optimal solutions and providing a numerical example. \\

\noindent{\bf Organization of the rest of the paper.} In Section~\ref{sec:prelim} we recall the optimal transport problem, the relevant graph theoretic notions, and important results from both optimal transport and graph theory. In Section~\ref{sec:graph_for_OT} we introduce the graph associated to the optimal transport problem and prove the first two main results of the paper: the uniqueness criterion and the characterization of the set of all optimizers. In Section~\ref{sec:entropic} we turn to the entropic regularization and provide the characterization of the limit of the regularized dual optimizers. Finally, in Section~\ref{sec:SCNE} we consider the game between a continuum of agents and a principal, describe existence results and connect the optimization problem of the principal to the problem of finding all dual optimizers for certain optimal transport problems. We conclude the section by providing approximation arguments and a numerical illustration.

\section{Preliminaries.}
\label{sec:prelim}
In this section we recall some fundamental results in optimal transport and in graph theory, that will be used throughout the paper.

\subsection{Finite Optimal Transport.}\label{sect.fOT}
Let $\mu$ and $\nu$ be two discrete probability measures on $\RR^d$, with finite supports $\Xcal$ and $\Ycal$ having cardinality $n_\Xcal$ and $n_\Ycal$, respectively. For a function $c:\Xcal\times \Ycal \to\RR_+$, the optimal transport problem between $\mu$ and $\nu$ with respect to the cost $c$ is given by
\begin{equation}\label{eq.OT}
\OT(\mu,\nu, c) = \inf_{\gamma \in \Pi(\mu, \nu)} \int c(x,y)\gamma(dx,dy),
\end{equation}
where $\Pi(\mu, \nu)\subseteq\Pcal(\Xcal\times\Ycal)$ is the set of probability measures on $\Xcal\times\Ycal$ with first marginal equal to $\mu$ and second marginal equal to $\nu$. 
An element $\gamma$ of $\Pi(\mu,\nu)$ is called a coupling of $\mu$ and $\nu$, and is called an \emph{optimal coupling}, or \emph{primal optimizer}, if it is an optimizer for problem \eqref{eq.OT}.   
The associated dual optimization problem reads as
\begin{align}\label{eq.DOT}
\DOT (\mu, \nu, c) = \sup  \Big\{ \textstyle{ \int_\Xcal \varphi d\mu + \int_\Ycal\psi d\nu : \varphi:\Xcal\to\RR, \psi:\Ycal\to\RR},\ \varphi(x) + \psi(y) \le c(x,y) \, \forall x \in \Xcal, y \in \Ycal\Big\}.
\end{align}
A pair $(\varphi,\psi)$ satisfying the constraints in \eqref{eq.DOT} is called \emph{feasible}, and referred to as \emph{dual optimizer} if it is an optimizer for problem \eqref{eq.DOT}.
Since the finite optimal transport problem is a linear optimization problem, we immediately see that both the set of all primal optimizers and the set of all dual optimizers are convex. 
We refer the reader to the manuscript of Villani~\cite{VillaniOldAndNew} for a thorough exposition of the optimal transport theory. We summarize below some of the crucial results about problems \eqref{eq.OT} and \eqref{eq.DOT} that will be useful for later reference. For this, we recall that a set $\Gamma\subseteq \Xcal\times\Ycal$ is called \emph{$c$-cyclically monotone} if, for any $N\in\mathbb{N}$ and any collection $(x_1,y_1),\ldots,(x_N,y_N)\in\Gamma$, the inequality
\[
\sum_{i=1}^N c(x_i,y_i)\leq \sum_{i=1}^N c(x_i,y_{i+1})
\]
holds, with the convention $y_{N+1}=y_1$. A coupling $\gamma\in\Pcal(\Xcal\times\Ycal)$ is said to be \emph{$c$-cyclically monotone} if it is concentrated on a $c$-cyclically monotone set.

\begin{theorem}[\cite{VillaniOldAndNew}, Theorem 5.10, Remark 5.12]\label{thm.510}
In the above discrete setting, we have:
\begin{itemize}
\item[(i)] duality holds, i.e.\
$\OT (\mu, \nu, c)=\DOT (\mu, \nu, c)$;
\item[(ii)] both the primal problem \eqref{eq.OT} and the dual problem \eqref{eq.DOT} admit solutions;
\item[(iii)] there is a $c$-cyclically monotone set $\Gamma\subseteq \Xcal\times\Ycal$ such that, for $\gamma\in\Pi(\mu,\nu)$, the following are equivalent:
\begin{itemize}
\item[1.] $\gamma$ is optimal for \eqref{eq.OT};
\item[2.] $\gamma$ is concentrated on $\Gamma$;
\item[3.] $\gamma$ is $c$-cyclically monotone;
\end{itemize}
\item[(iv)] let $\gamma\in\Pi(\mu,\nu)$, and $(\varphi,\psi)$ be a feasible pair for the dual problem, then $\gamma$ and $(\varphi,\psi)$ are optimal solutions for the primal resp. dual problem if and only if they are complementary, i.e.
\[
\varphi(x)+\psi(y)=c(x,y)\, \gamma\text{-a.s.};
\]
\item[(v)] the union of the supports of all primal optimizers is the smallest $c$-cyclically monotone set contained in $\Xcal\times\Ycal$  and such that all primal optimizers are concentrated on it.
\end{itemize}
\end{theorem}

\begin{remark}[uniqueness]
\label{remark:uniqueness}
It is immediate to see that if $(\varphi, \psi)$ is a dual optimizer, then also the pair obtained by translation, $(\varphi + a, \psi - a)$, $a\in\RR$, is a dual optimizer. One of the main results of the present paper consists in characterizing all dual optimizers, and for this we decide to adopt a normalization assumption. In light of this, we say that the dual optimizer is \emph{unique} if it is unique up to translation.
\hfill$\diamond$
\end{remark}

\subsection{Graph Theory.} 
\label{sec:GT}
A graph  is a pair $G=(V,E)$ of two sets, such that $V \neq \emptyset$ is a finite set and $E\subseteq [V]^2$, where $[V]^2$ is the set of all two-element subsets of $V$. Any element of $V$ is called a vertex. Moreover, we say that $e=\{v,w\}\in E$ is an edge and $v$ and $w$ are the end vertices of $e$. 
A graph $G'=(V',E')$ with $V' \subseteq V$ and $E' \subseteq E$ is called a subgraph of $G$, and in this case we write $G' \subseteq G$. We say that a graph $G$ is maximal with respect to some property if there is no graph $H\neq G$ with the same property and such that $G\subseteq H$. 
Given a graph $G=(V,E)$ and a set $U \subseteq V$, the induced subgraph $G[U]:=(U,E')\subseteq G$ is the graph such that $E'=\{\{v,w\} \in E: v,w \in U\}$, i.e. it contains all edges whose both end vertices lie in $U$.

A \emph{path} in $G=(V,E)$ is a sequence of vertices $P=v_0 \ldots v_l$  ($l \ge 0$) such that all $v_i$ are distinct and  $\{v_i,v_{i+1}\} \in E$ for all $i \in \{0, \ldots, l-1\}$. We call $v_0$ and $v_l$ the end vertices of the path $P$ and say that $P$ joins the vertices $v_0$ and $v_l$. Note that we allow paths of length zero, i.e. consisting of one vertex only. 
We say that a graph $G=(V,E)$ is \emph{connected} if any two of its vertices are linked by a path. We say that a set $U\subseteq V$ is connected in $G$ if the induced subgraph $G[U]$ is connected. 

\begin{prop}[\cite{DiestelGT}, Proposition 1.4.1]\label{lemma.ordering}
Let $G=(V,E)$ be a connected graph and let $v \in V$ be an arbitrary vertex. We can order the vertices of $G$ as $v_1, \ldots, v_n$ such that $v=v_1$ and $G[\{v_1, \ldots, v_i\}]$ is connected for all $i \in \{2, \ldots, n\}$.
\end{prop}

A maximal connected subgraph of $G$ is a \emph{component} of $G$. We highlight that components are induced subgraphs of $G$, that any component can be identified by its vertex set, and that the vertex sets $V_1, \ldots, V_N$ of all components of $G$ partition the vertex set of $G$. Hence, with a slight abuse of notation, we will say that $V_1, \ldots, V_N$ are the components of $G$.
We say that a graph $G=(V,E)$ is \emph{bipartite} if we can partition the vertex set into two classes $W_1$ and $W_2$ such that every edge has one end vertex in $W_1$ and the other one in $W_2$. If $G$ is a bipartite graph and $e=\{w_1,w_2\}\in E$, then we write $e=[w_1,w_2]$ to indicate that $w_1 \in W_1$ and $w_2 \in W_2$. 
A \emph{cycle} is a sequence $v_0 v_1\ldots v_{l-1} v_0$ such that $v_0\ldots v_{l-1}$ is a path and $\{v_{l-1},v_0\} \in E$. A connected graph without cycles is called a \emph{tree}.

\begin{example}
Figure~\ref{fig:graphExamples} shows examples of graphs: The graph in Figure~\ref{fig:bipartite} is  bipartite graph with vertex set $V=\{x_1, x_2, y_1,y_2\}$ and edge set $E=\{ [x_1,y_1], [x_1, y_2], [x_2, y_1], [x_2,y_2]\}$. It is bipartite  with classes $W_1=\{x_1,x_2\}$ and $W_2=\{y_1,y_2\}$. Moreover, it is connected since there is a path from every vertex to every other vertex. For example, a path from $x_1$ to $y_1$ is given by $x_1y_1$ and a path from $x_1$ to $x_2$ is given by $x_1y_1x_2$. The graph in Figure~\ref{fig:components}  is not connected since there is no path from $5$ to $1$. Instead, it has two components $V_1 = \{1,2,3,4\}$ and $V_2=\{5,6,7\}.$  The graph is not a tree because it contains the cycle $2,3,4$, nor bipartite because we cannot partition $\{2,3,4\}$ into two sets such that any edge goes from one set to the other. Finally, the graph in Figure~\ref{fig:tree} is a tree because it is connected and does not contain any cycles. 

\begin{figure}[h]
\centering
\subcaptionbox{A bipartite graph \label{fig:bipartite}}[0.2\textwidth]{
\begin{tikzpicture}
\clip (-1,-0.25) rectangle (2.75, 2.5);

\coordinate[label={[label distance=0.2cm]left:$x_2$}] (L1) at (0,0.5);
\fill (L1) circle (3pt);
\coordinate[label={[label distance=0.2cm]left:$x_1$}] (L2) at (0,1.5);
\fill (L2) circle (3pt);
\coordinate[label={[label distance=0.2cm]right:$y_2$}] (R1) at (1.5,0.5);
\fill (R1) circle (3pt);
\coordinate[label={[label distance=0.2cm]right:$y_1$}] (R2) at (1.5,1.5);
\fill (R2) circle (3pt);
\coordinate (R3) at (4,0);
\draw[thick] (R2) -- (L1) -- (R1);
\draw[thick] (R2) -- (L2) -- (R1);

\end{tikzpicture}}
\subcaptionbox{A graph with two components\label{fig:components}}[0.35\textwidth]{\begin{tikzpicture}
\clip (-0.55,-0.25) rectangle (3.75, 2.5);

\coordinate[label={[label distance=0.2cm]right:$1$}] (1) at (2,0.9);
\fill (1) circle (3pt);
\coordinate[label={[label distance=0.2cm]left:$2$}] (2) at (0,1.25);
\fill (2) circle (3pt);
\coordinate[label={[label distance=0.2cm]right:$3$}] (3) at (1.5,1.5);
\fill (3) circle (3pt);
\coordinate[label={[label distance=0.2cm]right:$4$}] (4) at (2.5,2.25);
\fill (4) circle (3pt);
\coordinate[label={[label distance=0.2cm]left:$5$}] (5) at (0.5,0.25);
\fill (5) circle (3pt);
\coordinate[label={[label distance=0.2cm]right:$6$}] (6) at (2,0.25);
\fill (6) circle (3pt);
\coordinate[label={[label distance=0.2cm]right:$7$}] (7) at (3,1);
\fill (7) circle (3pt);

\draw[thick] (1) -- (2) -- (3) -- (4) -- (2);
\draw[thick] (5) -- (6) -- (7);

\end{tikzpicture}}		
\subcaptionbox{A tree 
\label{fig:tree}}[0.3\textwidth]{\begin{tikzpicture}
\clip (-0.55,-0.25) rectangle (3.75, 2.5);

\coordinate[label={[label distance=0.2cm]left:$1$}] (R) at (1.5,0);
\fill (R) circle (3pt);
\coordinate[label={[label distance=0.2cm]left:$2$}] (Ba1) at (1,1);
\fill (Ba1) circle (3pt);
\coordinate[label={[label distance=0.2cm]left:$4$}] (Ba21) at (0,2);
\fill (Ba21) circle (3pt);
\coordinate[label={[label distance=0.2cm]left:$5$}] (Ba22) at (1.75,2);
\fill (Ba22) circle (3pt);
\coordinate[label={[label distance=0.2cm]left:$3$}] (Bb1) at (3,1);
\fill (Bb1) circle (3pt);
\coordinate[label={[label distance=0.2cm]left:$6$}] (Bb2) at (3,2);
\fill (Bb2) circle (3pt);
\draw[thick] (Ba1) -- (R) -- (Bb1) -- (Bb2);
\draw[thick] (Ba22) -- (Ba1) -- (Ba21);

\end{tikzpicture}}
\caption{Examples of graphs}
\label{fig:graphExamples}
\end{figure}
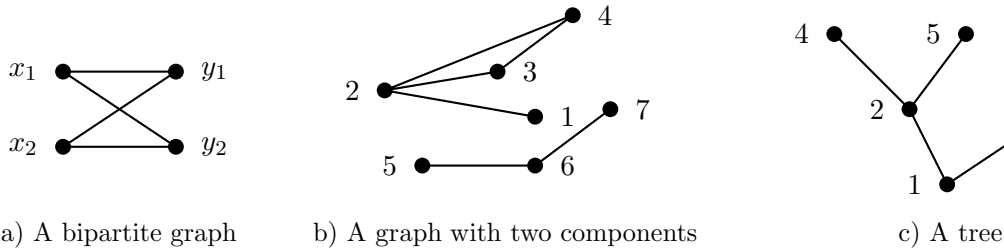
\hfill$\diamond$
\end{example}

\section{Connectivity and Transport Optimizers.}
\label{sec:graph_for_OT}
In this section we introduce graphs that allow to derive structural results on the optimal transport problem. In particular, we provide a necessary and sufficient criterion for uniqueness of the dual optimizer and a characterization of all dual optimizers given one primal optimizer. 
For this, we fix $\mu,\nu$ and $c$ as in Section~\ref{sect.fOT}, and consider the corresponding primal and dual problems \eqref{eq.OT} and \eqref{eq.DOT}.

Let $\gamma$ be a primal optimizer of $\OT(\mu, \nu, c)$. Then we define the bipartite graph $G^\gamma =(V^\gamma, E^\gamma)$ with vertex set $V^\gamma = \Xcal \cup \Ycal$ (with partition $W_1=\Xcal$ and $W_2=\Ycal$) and edge set 
\begin{align*}
E^\gamma = \Big\{ [x,y] : \gamma(x,y)>0 \Big\}.
\end{align*}
Moreover, we consider the bipartite graph $G=(V,E)$ with vertex set $V= \Xcal\cup\Ycal$ and edge set given by the union of all $E^\gamma$ with $\gamma$ optimal for \eqref{eq.OT}, that is
\begin{equation}\label{eq.setE}
E = \Big\{ [x,y] : \gamma(x,y)>0 \text{ for some primal optimizer $\gamma$}\Big\}.
\end{equation} 
We call $G$ the graph associated to $\OT(\mu,\nu,c)$.

\begin{remark}\label{rem.supp}
Note that the set $E$ in \eqref{eq.setE} corresponds to the set described in Theorem~\ref{thm.510}-(v). Moreover, by Theorem~\ref{thm.510}-(iii),(v), $\gamma \in \Pi(\mu, \nu)$ is a primal optimizer if and only if 
\[
\gamma(x,y)=0 \quad \text{for all } [x,y] \notin E.
\]
\hfill $\diamond$
\end{remark}

\begin{example}
\label{ex1}
Let us illustrate the definitions as well as the differences between the sets $G^\gamma$ and $G$ introduced above in a simple example. Consider $\Xcal = \{1,2,3\}$ and $\Ycal = \{1,2,3,4\}$ and the measures
\[
\mu = \tfrac{1}{4} \delta_{\{1\}} + \tfrac{1}{4} \delta_{\{2\}} + \tfrac{1}{2} \delta_{\{3\}} \quad \text{and} \quad \nu = \tfrac{1}{10} \delta_{\{1\}} + \tfrac{2}{5} \delta_{\{2\}} + \tfrac{1}{5} \delta_{\{3\}} + \tfrac{3}{10} \delta_{\{4\}}.
\] 
Set
$F=\left\{ (1,1), (1,2), (2,1), (2,2), (2,3), (3,2), (3,3), (3,4) \right\}$     
and define
$ c=1_{F^c}$.
By definition of $c$, we have $\OT(\mu, \nu,c)\ge 0$. Since $\gamma: \Xcal \times \Ycal \rightarrow [0,1]$ with
\[
\gamma(1,1) = \tfrac{1}{10}, \, \gamma(1,2) = \tfrac{3}{20}, \, \gamma(2,2) = \tfrac{1}{4}, \, \gamma(3,3) = \tfrac{1}{5}, \, \gamma(3,4) = \tfrac{3}{10}
\] 
is feasible and has cost $0$, we have $\text{OT}(\mu, \nu, c) =0$. The graph $G^\gamma$ is depicted in Figure~\ref{fig:ex1_Ggamma}. 
It is clear that any feasible coupling $\gamma$ is optimal if and only if, for any $(x,y) \in \Xcal \times \Ycal$, $\gamma(x,y) >0 \Rightarrow (x,y) \in F$. Moreover, one can easily see that there exists a coupling  with support $F$, which implies $E=F$. The graph $G$ is depicted in Figure~\ref{fig:ex1_G}. Note that $G^\gamma$ is a proper subgraph of $G$, and that it is not connected although $G$ is.

\begin{figure}[h]
\centering
\subcaptionbox{$G^\gamma$ \label{fig:ex1_Ggamma}}[0.2\textwidth]{
\begin{tikzpicture}

\coordinate[label={[label distance=0.2cm]left:$3$}] (L3) at (0,1.5);
\fill (L3) circle (3pt);
\coordinate[label={[label distance=0.2cm]left:$2$}] (L2) at (0,2.5);
\fill (L2) circle (3pt);
\coordinate[label={[label distance=0.2cm]left:$1$}] (L1) at (0,3.5);
\fill (L1) circle (3pt);
\coordinate[label={[label distance=0.2cm]right:$4$}] (R4) at (2,0.5);
\fill (R4) circle (3pt);
\coordinate[label={[label distance=0.2cm]right:$3$}] (R3) at (2,1.5);
\fill (R3) circle (3pt);
\coordinate[label={[label distance=0.2cm]right:$2$}] (R2) at (2,2.5);
\fill (R2) circle (3pt);
\coordinate[label={[label distance=0.2cm]right:$1$}] (R1) at (2,3.5);
\fill (R1) circle (3pt);
\draw[thick] (L2) -- (R2) -- (L1) -- (R1);
\draw[thick] (R3) -- (L3) -- (R4);

\end{tikzpicture}}	
\hspace*{1.5cm}
\subcaptionbox{$G$ \label{fig:ex1_G}}[0.25\textwidth]{
\begin{tikzpicture}

\coordinate[label={[label distance=0.2cm]left:$3$}] (L3) at (0,1.5);
\fill (L3) circle (3pt);
\coordinate[label={[label distance=0.2cm]left:$2$}] (L2) at (0,2.5);
\fill (L2) circle (3pt);
\coordinate[label={[label distance=0.2cm]left:$1$}] (L1) at (0,3.5);
\fill (L1) circle (3pt);
\coordinate[label={[label distance=0.2cm]right:$4$}] (R4) at (2,0.5);
\fill (R4) circle (3pt);
\coordinate[label={[label distance=0.2cm]right:$3$}] (R3) at (2,1.5);
\fill (R3) circle (3pt);
\coordinate[label={[label distance=0.2cm]right:$2$}] (R2) at (2,2.5);
\fill (R2) circle (3pt);
\coordinate[label={[label distance=0.2cm]right:$1$}] (R1) at (2,3.5);
\fill (R1) circle (3pt);
\draw[thick] (L2) -- (R2) -- (L1) -- (R1) -- (L2) --(R1);
\draw[thick] (L2) -- (R3) -- (L3) -- (R4);	
\draw[thick] (R2) -- (L3);

\end{tikzpicture}}
\caption{The graphs $G^\gamma$ and $G$ from Example~\ref{ex1}}
\label{fig:ex1}
\end{figure}
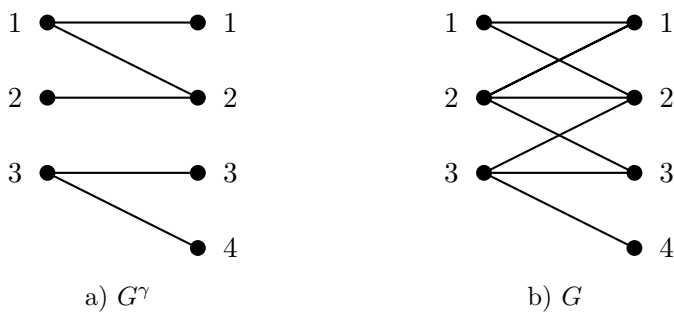
\hfill$\diamond$
\end{example}

\begin{lemma}
\label{lemma:GvsGgamma}
There is a primal optimizer $\gamma^\ast$ such that $G=G^{\gamma^\ast}$.
\end{lemma}

\begin{proof}
Since the graph $G$ has finitely many vertices, it also has finitely many edges $e_1, \ldots, e_M$, with $e_m=[x_m,y_m]$, $x_m\in\Xcal$ and $y_m\in\Ycal$, for each $m\in\{1,\ldots,M\}$. By construction of $G$, there is for each $m \in \{1, \ldots, M\}$ a primal optimizer $\gamma^m$ such that $\gamma^m(x_m,y_m)>0$. Since the set of all primal optimizers is convex, the coupling defined as
\[
\gamma^\ast = \frac{1}{M} \sum_{m=1}^M \gamma^m
\]	 
is a primal optimizer as well. Moreover, it satisfies 
\[
\gamma^\ast(x_m,y_m) \ge \frac{1}{M} \gamma^m(x_m,y_m)>0\quad \text{for all } m \in \{1, \ldots, M\}.
\] 
Hence, $G^{\gamma^\ast}=G$.
\end{proof}

Next we show that the graph $G$ also characterizes dual optimizers.

\begin{prop}
\label{prop:Gcharacterizes}
Let $(\varphi, \psi)$ be feasible for the dual problem. Then it is a dual optimizer if and only if 
\begin{equation}\label{eq.ug_E}
\varphi(x) + \psi(y) = c(x,y) \quad \text{for all } [x,y] \in E.
\end{equation}
\end{prop}

\begin{proof}
Let $\gamma^\ast$ be the optimizer from Lemma~\ref{lemma:GvsGgamma}. Then $G=G^{\gamma^\ast}$ and the support of $\gamma^*$ is $E$, i.e. the union of the supports of all primal optimizers. 
Therefore, if $\varphi(x) + \psi(y) = c(x,y)$ for all $[x,y] \in E$, then $\varphi(x) + \psi(y) = c(x,y)$ $\gamma^\ast-$a.s. Hence, $\gamma^\ast$ and $(\varphi, \psi)$ are complementary, which by Theorem~\ref{thm.510} means that $(\varphi, \psi)$ is a dual optimizer.

Vice versa, if $(\varphi, \psi)$ is a dual optimizer, then, by Theorem~\ref{thm.510}, $\gamma^\ast$ and $(\varphi,\psi)$ have to be complementary. This exactly means that $\varphi(x) + \psi(y) = c(x,y)$ for all $[x,y] \in E$.
\end{proof}

\subsection{Uniqueness.}
In this section, we provide sufficient conditions for the dual optimizer to be unique (meaning uniqueness up to translation, see Remark~\ref{remark:uniqueness}). 

\begin{prop}
\label{prop:unique}
Let $G$ be the graph associated to $\OT(\mu,\nu,c)$. Then:
\begin{itemize}
\item[(i)] if for all non-empty subsets $X\subsetneq \Xcal$ and $Y \subsetneq \Ycal$ we have $\mu(X)\neq \nu(Y)$, then the graph $G$ is connected;
\item[(ii)] if the graph $G$ is connected, then the dual optimizer is unique.
\end{itemize} In particular, the dual optimizer is unique whenever for all non-empty subsets $X\subsetneq \Xcal$ and $Y \subsetneq \Ycal$ we have $\mu(X)\neq \nu(Y)$.
\end{prop}	

The last statement of the proposition has already been described by Staudt et al. \cite{StaudtNonUniqueness}. The sufficient condition in (ii) is new, and we will see later in Corollary~\ref{cor:UniqueMeansConnected} that is actually necessary.

\begin{proof}
(i): We are going to show that, for any primal optimizer $\gamma$, the graph $G^\gamma$ is connected. This in turn implies the claim, by considering $\gamma^\ast$ from Lemma~\ref{lemma:GvsGgamma}. 
Assume that there is a primal optimizer $\gamma$ such that $G^\gamma$ is not connected. In this case there are two non-empty sets $U_1$ and $U_2$ such that $U_1 \cap U_2 = \emptyset$, $U_1 \cup U_2 = \Xcal \cup \Ycal$ and that there is no edge from $U_1$ to $U_2$ in $E^\gamma$, i.e., for all $(e_1,e_2) \in E^\gamma$ we have that either $e_1, e_2 \in U_1$ or $e_1,e_2 \in U_2$.
By construction of $G^\gamma$, we can write $U_1 = X_1 \cup Y_1$ and $U_2=X_2 \cup Y_2$ for some subsets $X_1, X_2 \subseteq \Xcal$ and $Y_1, Y_2 \subseteq \Ycal$. Note that all subsets $X_1,X_2,Y_1,Y_2$ are non-empty. Indeed, suppose by contradiction that for example $Y_1=\emptyset$. Then we would have $X_1=U_1\neq\emptyset$. Since $\mu(x)>0$ for all $x\in\Xcal$, then in particular for any $x\in X_1$ there is $y\in\Ycal$ such that $\gamma(x,y)>0$. By the above decomposition of $E^\gamma$ we would then need to have $y\in U_1$ as well, that is $y\in Y_1$, which leads to the desired contradiction.
Similarly, we would find a contradiction by assuming any of the other sets to be empty.

Now note that, since there are no edges from $U_1$ to $U_2$, then
\[
\gamma(x,y)=0 \text{ for all } (x,y) \in(X_1 \times Y_2) \cup (X_2 \times Y_1).
\] 
Together with the fact that $\gamma\in\Pi(\mu,\nu)$, this yields
\begin{align*}
\mu(X_1) = \sum_{x \in X_1} \mu(x) = \sum_{x \in X_1} \sum_{y \in \Ycal} \gamma(x,y) = \sum_{x \in X_1} \sum_{y \in Y_1} \gamma(x,y) = 
\sum_{y \in Y_1} \sum_{x \in X_1}  \gamma(x,y)
= \sum_{y \in Y_1} \nu(y) = \nu(Y_1),
\end{align*} 
which is a contradiction.

(ii): To prove uniqueness up to translation, we fix an arbitrary $x_0 \in \Xcal$ and show that there is exactly one dual optimizer with $\varphi(x_0)=0$.
By Proposition~\ref{lemma.ordering}, there is an ordering $v_1, v_2, \ldots, v_{n_\Xcal+n_\Ycal}$ of the vertices of $G$, with $v_1=x_0$ and such that $G[\{v_1, \ldots, v_i\}]$ is connected for all $i \in \{2, \ldots, n_\Xcal+n_\Ycal\}$.  Consider any dual pair $(\varphi,\psi)$ with $\varphi(v_1)=0$. Since $G[\{v_1, v_2\}]$ is connected, then $[v_1,v_2]\in E$ and, by Proposition~\ref{prop:Gcharacterizes}, we have that $\psi(v_2)=c(v_1,v_2)-\varphi(v_1)$. In the same way, we argue that since $G[\{v_1, v_2, v_3\}]$ is connected, then either $[v_1,v_3]\in E$ or $[v_3,v_2]\in E$. In the first case, $v_3\in\Ycal$ and Proposition~\ref{prop:Gcharacterizes} yields $\psi(v_3)=c(v_1,v_3)-\varphi(v_1)$, while in the second case we have $v_3\in\Xcal$ and $\varphi(v_3)=c(v_3,v_2)-\psi(v_2)$. By iterating this process, we see that all values $\varphi(x), x\in\Xcal$, and $\psi(y), y\in\Ycal$, are uniquely identified.
Hence, the dual optimizer $(\varphi, \psi)$ is unique up to translation.
\end{proof}

Note that the proof of the second part of Proposition~\ref{prop:unique} also describes a way to determine the unique optimizer given the graph $G$. Namely, it suffices to find an ordering $v_1, \ldots, v_{n_\Xcal + n_\Ycal}$ of the vertices of $G$ such that $v_1 = x_0$ and $G[\{v_1, \ldots, v_i\}]$ is connected for all $i \in \{2, \ldots, n_\Xcal+n_\Ycal\}$, and thereafter successively set the values of $\varphi$ resp. $\psi$ according to \eqref{eq.ug_E}.

\begin{continueexample}{ex1}
Since $G$ is connected, there is a unique dual optimizer.
Note that $\mu(\{1,2\}) = \nu(\{1,2\})$, which means that the condition that $\mu(X)\neq\nu(Y)$ for all $X \subsetneq \Xcal$ and $Y \subsetneq \Ycal$ is not necessary for uniqueness.
Moreover, as explained before, we can easily compute the dual optimizer: an ordering satisfying the desired conditions is $1_\Xcal, 1_\Ycal, 2_\Xcal, 2_\Ycal, 3_\Xcal, 3_\Ycal, 4_\Ycal$, where indexes here are to show the belonging set.
Hence, following the method described in the proof of Proposition~\ref{prop:unique}-(ii) and starting with $\varphi(1)=0$, we find that $(\varphi,\psi)$ with $\varphi\equiv 0$ and $\psi\equiv 0$ is the unique dual optimizer. 
\hfill $\diamond$
\end{continueexample}

For a connected graph $G$, the unique dual optimizer immediately characterizes the graph $G$. Later we will see that a weaker statement holds also for general (not necessarily connected) graphs. Namely, we can show that there are some dual optimizers that satisfy \eqref{eq:CharacterizationGConnected} below, while there are always some dual optimizers for which $E \subsetneq \{[x,y]: x\in\Xcal, y\in\Ycal, c(x,y) = \varphi(x) + \psi(y)\}$ (see Corollary~\ref{cor:relationEandPhi}).

\begin{prop}
\label{prop:ChracterizationGconnected}
Assume that the graph $G$ associated to $\OT(\mu,\nu,c)$ is connected and let $(\varphi, \psi)$ be the unique dual optimizer. Then
\begin{equation}
\label{eq:CharacterizationGConnected}
E = \{[x,y]: x\in\Xcal, y\in\Ycal, c(x,y) = \varphi(x) + \psi(y)\}.
\end{equation}
\end{prop}

\begin{proof}
By Proposition~\ref{prop:Gcharacterizes}, we have that $E \subseteq \{[x,y]: x\in\Xcal, y\in\Ycal, c(x,y) = \varphi(x) + \psi(y)\}$. To prove the other inclusion, assume by contradiction that there are $\hat x\in\Xcal$ and $\hat y\in\Ycal$ such that $[\hat{x},\hat{y}] \notin E$ while $c(\hat{x}, \hat{y}) = \varphi(\hat{x}) + \psi(\hat{y})$. 
We are now going to construct a primal optimizer $\hat{\gamma}$ such that $\hat{\gamma}(\hat{x},\hat{y})>0$, which yields the desired contradiction.
Since $G$ is connected and bipartite, there is a path  $x_1 y_1 x_2 y_2 \ldots x_l y_l$ from $\hat{x}=x_1$ to $\hat{y}=y_l$ with $x_1,\ldots x_l\in\Xcal$ and $y_1,\ldots,y_l\in\Ycal$. Now let $\gamma$ be a primal optimizer such that $G=G^\gamma$. Note that this implies $\gamma(x_j,y_j)>0$ for $j\in\{1,\ldots,l\}$, and $\gamma(x_j,y_{j-1})>0$ for $j\in\{2,\ldots,l\}$.
Then we set
\[
\delta =   \min_{j \in \{1, \ldots, l\}} \gamma(x_j,y_j)\ \wedge \min_{j \in \{2, \ldots, l\}} \big(1- \gamma(x_{j},y_{j-1})\big).
\] 
Note that $\delta \in(0,1)$. We define
\[
\hat{\gamma}(x,y) =\begin{cases}
\gamma(x,y) - \delta &\text{if } (x,y) = (x_j,y_j) \text{ for some } j \in \{1, \ldots, l\} \\
\gamma(x,y) + \delta &\text{if } (x,y) = (x_j, y_{j-1}) \text{ for some } j \in \{1, \ldots, l\} \\
\gamma(x,y) &\text{else},
\end{cases}
\]
with the convention $y_0:=y_l$.

Now we show that $\hat{\gamma} \in \Pi(\mu, \nu)$. First, note that $\hat{\gamma}(x,y) \in [0,1]$ for all $(x,y) \in \Xcal \times \Ycal$ by choice of $\delta$. Now, let $x \in \Xcal \setminus \{x_1, \ldots, x_l\}$. Then
\[
\sum_{y \in \mathcal{Y}} \hat{\gamma}(x,y) = \sum_{y \in \Ycal} \gamma(x,y) = \mu(x).
\] Now assume that $x= x_j$ for some $j \in \{1, \ldots, l\}$. Then we obtain
\begin{align*}
\sum_{y \in \Ycal} \hat{\gamma}(x_j,y) = \sum_{y \in \Ycal \setminus \{y_j,y_{j-1}\}}\gamma(x_j,y) + \gamma(x_j,y_j) - \delta + \gamma(x_j, y_{j-1}) + \delta = \sum_{y \in \Ycal} \gamma(x_j,y)=\mu(x_j).
\end{align*} 
Analogous computations for $y \in \Ycal$ show that $\hat{\gamma} \in \Pi(\mu,\nu)$. 

To show that $\hat{\gamma}$ and $(\varphi, \psi)$ are complementary, it suffices to note that, for all $(x,y) \neq (\hat{x},\hat{y})$, we have
\[
\hat{\gamma}(x,y) >0 \Rightarrow \gamma(x,y) >0.
\] 
Hence, the complementarity of $\gamma$ and $(\varphi, \psi)$, as well as the choice of $\hat{x}$ and $\hat{y}$, yield that $\hat{\gamma}$ and $(\varphi, \psi)$ are complementary. By Theorem~\ref{thm.510}-(iv) we therefore obtain that $\hat{\gamma}$ is a primal optimizer, which yields $[\hat{x}, \hat{y}] \in E$. This is the desired contradiction.
\end{proof}

\subsection{Optimal Transport on Connected Components.}
In this section we show that the optimal transport problem on $\Xcal \times \Ycal$ is closely related to the optimal transport problems restricted to the components of the relevant graph. For this we introduce the following notation.
Let $X\subseteq\Xcal$ and $Y\subseteq\Ycal$ be such that $\mu(X) = \nu(Y)$. Then we denote the transport problem restricted to $X\times Y$ by
\[
\OT_{X,Y}(\mu,\nu,c):=\OT \left( \frac{1}{\mu(X)} \mu|_{X}, \frac{1}{\nu(Y)} \nu|_{Y}, c|_{X \times Y} \right),
\]
and use the notation $\DOT_{X,Y}(\mu,\nu,c)$ for its dual problem.

\begin{theorem}
\label{thm:Components}
Let $\gamma$ be a primal optimizer of $\OT(\mu,\nu,c)$, and let $V_1, \ldots, V_N$ be the connected components of $G^\gamma$, 
with $V_n=X_n \cup Y_n$ for subsets $X_n \subseteq \Xcal$ and $Y_n\subseteq \Ycal$ for $n \in \{1, \ldots, N\}$. Then:
\begin{itemize}
\item[(i)] for all $n,m \in \{1, \ldots, N\}$ with $n \neq m$, $\gamma(X_n\times Y_m) =0$;
\item[(ii)] for all $n \in \{1, \ldots, N\}$, $\frac{1}{\mu(X_n)} \gamma|_{X_n \times Y_n}$ is a primal optimizer for $\OT_{X_n,Y_n}(\mu,\nu,c)$;
\item[(iii)] it holds that
\begin{align*}
\OT(\mu,\nu, c) = \sum_{n=1}^N \mu(X_n) \OT_{X_n,Y_n}(\mu,\nu,c);
\end{align*}
\item[(iv)] if $(\varphi, \psi)$ is an optimizer for the dual problem $\DOT(\mu,\nu, c)$, then, for all $n \in \{1, \ldots, N\}$,  $(\varphi|_{X_n}, \psi|_{Y_n})$ is an optimizer for the dual problem $\DOT_{X_n,Y_n}(\mu,\nu,c)$.
\end{itemize}
\end{theorem}

\begin{proof}
(i): This is clear by definition of $G^\gamma$.

(ii): This is an immediate consequence of Theorem 4.6 in Villani~\cite{VillaniOldAndNew}, which asserts, in the current discrete setting, that if $\gamma$ is a primal optimizer and $\gamma'$ is a non-negative measure on $\Xcal\times\Ycal$ such that $\gamma' \le \gamma$ with $\gamma'(\Xcal \times \Ycal)>0$, then $\gamma'/\gamma'(\Xcal \times \Ycal)$ is optimal for $\OT(\mu',\nu',c)$ where $\mu'$ and $\nu'$ are the marginals of $\gamma'/\gamma'(\Xcal \times \Ycal)$. In our case, for any $n\in\{1,\ldots,N\}$, we choose 
\[
\gamma' (x,y) = \begin{cases}
\gamma(x,y) &\text{if }x \in X_n, y \in Y_n \\
0 &\text{otherwise}
\end{cases}
\] 
and note that, by (i),
\[
\gamma'(\Xcal \times \Ycal) = \gamma(X_n \times Y_n) = \sum_{x \in X_n, y \in Y_n} \gamma(x,y) = \sum_{x \in X_n, y \in \Ycal} \gamma(x,y) = \mu(X_n)>0.
\] 
Finally, using (i) we have
\begin{align*}
\mu'(x) &= \frac{1}{\gamma'(\Xcal \times \Ycal)} \sum_{y \in Y_n} \gamma'(x,y) = \frac{1}{\mu(X_n)}\sum_{y \in \Ycal} \gamma(x,y) =  \frac{1}{\mu(X_n)} \mu(x) \quad \text{for } x \in X_n, \\
\nu'(y) &=  \frac{1}{\gamma'(\Xcal \times \Ycal)} \sum_{x \in X_n} \gamma'(x,y) = \frac{1}{\nu(Y_n)}\sum_{x \in \Xcal} \gamma(x,y) =  \frac{1}{\nu(Y_n)}\nu(y) \quad \text{for } y \in Y_n, 
\end{align*} 
which shows that 
$\frac{1}{\mu(X_n)} \gamma|_{X_n \times Y_n}$ is a primal optimizer for $\OT \left( \frac{1}{\mu(X_n)} \mu|_{X_n}, \frac{1}{\nu(Y_n)} \nu|_{Y_n}, c|_{X_n \times Y_n} \right)$.

(iii): By optimality of $\gamma$ and from (i), we have
\[
 \OT(\mu, \nu, c)=\int_{\Xcal \times \Ycal} c\ d\gamma = \sum_{n=1}^N \int_{X_n\times Y_n} c|_{X_n \times Y_n} \  d\gamma|_{X_n \times Y_n}.
\] 
Then by (ii), for all $n \in \{1, \ldots, N\}$,
\[
\frac{1}{\mu(X_n)} \int_{X_n\times Y_n} c|_{X_n \times Y_n}\ d  \gamma|_{X_n \times Y_n} = \OT_{X_n,Y_n}(\mu,\nu,c).
\] 
Combining the two equations gives the desired statement.

(iv): To show that $(\varphi|_{X_n}, \psi|_{Y_n})$ is an optimizer for $\DOT_{X_n,Y_n}(\mu,\nu,c)$, by Theorem~\ref{thm.510}-(iv) it suffices to show that $\frac{1}{\mu(X_n)} \gamma|_{X_n \times Y_n}$ and $(\varphi|_{X_n}, \psi|_{Y_n})$ are complementary. Note that,  for all $x \in X_n$ and $y \in Y_n$, we have
\[
\frac{1}{\mu(X_n)} \gamma|_{X_n \times Y_n} (x,y)>0 \Leftrightarrow \gamma(x,y)>0.
\] 
Since $\gamma$ and $(\varphi,\psi)$ are primal resp. dual optimizers, by Theorem~\ref{thm.510}-(iv) they are complementary, thus the statement follows.
\end{proof}

Given that for any primal optimizer $\gamma$ we have $E^\gamma \subseteq E$, the above theorem immediately allows us to draw conclusions about the graph $G$, and to characterize all primal optimizers given $G$.

\begin{corollary}
\label{cor:Components}
Let $G$ be the graph associated to $\OT(\mu,\nu,c)$ and $V_1, \ldots, V_N$ its connected components, 
with $V_n=X_n \cup Y_n$ for subsets $X_n \subseteq \Xcal$ and $Y_n\subseteq \Ycal$ for  $n \in \{1, \ldots, N\}$. Then:
\begin{itemize}
\item[(i)] any primal optimizer $\gamma$ satisfies $\gamma(X_n\times Y_m) = 0$ for all $n,m \in \{1, \ldots, N\}$ with $n \neq m$;
\item[(ii)] $\gamma \in \Pi(\mu, \nu)$ is a primal optimizer if and only if $\frac{1}{\mu(X_n)} \gamma|_{X_n \times Y_n}$ is optimal for the problem 
$\OT_{X_n,Y_n}(\mu,\nu,c)$
for all $n=1, \ldots, N$;
\item[(iii)] any dual optimizer $(\varphi, \psi)$ satisfies that, for all $n \in \{1, \ldots, N\}$, the vector $(\varphi|_{X_n}, \psi|_{Y_n})$ is an optimizer for $\DOT_{X_n,Y_n}(\mu,\nu,c)$.
\end{itemize}
\end{corollary}

\begin{proof} (i) and (iii) are immediate consequences of Theorem~\ref{thm:Components}, and so is the ``only if" implication in (ii). To see the converse implication, let $\hat{\gamma}$ be a primal optimizer such that $G^{\hat{\gamma}}=G$, which exists by Lemma~\ref{lemma:GvsGgamma}. Then Theorem~\ref{thm:Components}-(iii) applied to $\hat\gamma$ gives
\begin{align*}
\OT(\mu,\nu, c) = \sum_{n=1}^N \mu(X_n) \OT_{X_n,Y_n}(\mu,\nu,c).
\end{align*} 
Now fix any $\gamma\in\Pi(\mu,\nu)$ such that $\frac{1}{\mu(X_n)} \gamma|_{X_n \times Y_n}$ is optimal for $\OT_{X_n,Y_n}(\mu,\nu,c)$ for all $n\in\{1,\ldots,N\}$.
This implies in particular that $\gamma(x,y)=0$ whenever $x \in X_n$ and $y \in Y_m$ with $n \neq m$. Therefore
\begin{align*}
\int_{\Xcal\times\Ycal} c\ d\gamma = \sum_{n=1}^N \int_{X_n \times Y_n}c|_{X_n \times Y_n}\ d\gamma|_{X_n \times Y_n} = \sum_{n=1}^N \mu(X_n) \OT_{X_n,Y_n}(\mu,\nu,c)= \OT(\mu,\nu,c), 
\end{align*} 
which shows that $\gamma$ is indeed optimal. 
\end{proof}

\subsection{Characterization of all Dual Optimizers.}
Relying on the results obtained in the previous sections, we can now characterize the set of all dual optimizers. 

\begin{theorem}
\label{thm:SetOfAllDuals}
Let $\gamma$ be a primal optimizer for $\OT(\mu,\nu,c)$, and let $V_1, \ldots, V_N$ be the connected components of $G^\gamma$,  
with $V_n=X_n \cup Y_n$ for subsets $X_n \subseteq \Xcal$ and $Y_n\subseteq \Ycal$ for  $n \in \{1, \ldots, N\}$. Let $(\varphi_n,\psi_n)$ be the unique optimizer for $\DOT_{X_n,Y_n}(\mu, \nu, c)$ for $n \in \{1, \ldots, N\}$.
Then a pair $(\varphi, \psi)$ is an optimizer for $\DOT(\mu,\nu,c)$ if and only if there are constants $\alpha_1, \alpha_2, \ldots, \alpha_N \in \mathbb{R}$ such that:
\begin{itemize}
\item[(i)] for all $x \in X_n$, we have $\varphi(x) = \varphi_n(x) + \alpha_n + \alpha_1$ if $n\neq 1$, and $\varphi(x)=\varphi_1(x)+\alpha_1$ if $n=1$;
\item[(ii)] for all $y \in Y_n$, we have $\psi(y) = \psi_n(y)  - \alpha_n - \alpha_1$ if $n\neq 1$, and $\psi(y)=\psi_1(y)-\alpha_1$ if $n=1$;
\item[(iii)] for all $n, m \in \{1 \ldots, N\}$ with $n \neq m$, we have
\begin{equation}\label{eq.anm}
- \min_{x \in X_m, y \in Y_n} c(x,y) -\varphi_m(x)- \psi_n(y)\le \alpha_n - \alpha_m \le  \min_{x \in X_n, y \in Y_m} c(x,y) - \varphi_n(x) -\psi_m(y).  
\end{equation}
\end{itemize}
\end{theorem}

\begin{proof}
We start by showing that, for any dual optimizer $(\varphi, \psi)$, constants $\alpha_1, \alpha_2, \ldots, \alpha_N$ satisfying (i)-(iii) exist. By Theorem~\ref{thm:Components}-(iii) we know that, for every $n\in\{1,\ldots,N\}$, $\varphi|_{X_n}$ and $\psi|_{Y_n}$ have to be optimizers of $\DOT_{X_n,Y_n}(\mu, \nu,c)$. Since the optimizers for these problems are unique up to translation, there is exactly one set of constants $\alpha_1, \alpha_2,\ldots, \alpha_N$ such that the first two conditions are satisfied. Hence, it remains to prove that the constants $\alpha_1, \alpha_2, \ldots, \alpha_N$ given in this way satisfy (iii). For this, fix any $n,m \in \{1, \ldots, N\}$ with $n\neq m$.  Let us first consider any $x \in X_n$ and $y \in Y_m$. Since $(\varphi, \psi)$ is feasible and $(\varphi, \psi)$ satisfies (i) and (ii), then
\begin{align*}
c(x,y) &\ge \varphi(x) + \psi(y) = \varphi_n(x) + \alpha_n + \psi_m(y) - \alpha_m .
\end{align*} Hence, $\alpha_n - \alpha_m \le  c(x,y) - \varphi_n(x) - \psi_m(y)$, which proves the second inequality in (iii).

Similarly, let us now consider $x \in X_m$ and $y \in Y_n$. Then 
\begin{align*}
c(x,y) &\ge \varphi(x) + \psi(y) = \varphi_m(x) + \alpha_m + \psi_n(y) - \alpha_n,
\end{align*} 
thus $\alpha_n - \alpha_m \ge - c(x,y) + \varphi_m(x) + \psi_n(y)$, which proves the first inequality in (iii).

We are left to show the converse implication, that is, that given constants $\alpha_1, \ldots, \alpha_N$ satisfying (iii), the pair $(\varphi,\psi)$ defined via (i) and (ii) is a dual optimizer. 
For this we first prove that $(\varphi, \psi)$ is feasible, and then that it satisfies $c(x,y) = \varphi(x) + \psi(y)$ $\gamma-$a.s.. This means that $\gamma$ and $(\varphi,\psi)$ are complementary, which by Theorem~\ref{thm.510} implies optimality of $(\varphi,\psi)$.

To show that $(\varphi, \psi)$ is feasible, we first note that for $n \in \{1, \ldots, N\}$, $x \in X_n$ and $y \in Y_n$ we have
\[
\varphi(x) + \psi(y) =  \varphi_n(x)+\psi_n(y) \le c(x,y),
\] 
since $(\varphi_n,\psi_n)$ is an optimizer for $\DOT_{X_n,Y_n}(\mu, \nu, c)$. 
Now let  $n,m \in \{1, \ldots, N\}$ with $n \neq m$, and consider any  $x \in X_n$ and $y \in Y_m$. By choice of $\alpha_n$ and $\alpha_m$ we have
\[
\alpha_n - \alpha_m \le c(x,y) -\varphi_n (x) - \psi_m(y).
\]
Hence,
\begin{align*}
\varphi(x) + \psi(y) &= \varphi_n(x)  + \alpha_n + \psi_m(y)  -\alpha_m   \\
&\le \varphi_n(x) + \psi_m(y) + c(x,y) - \varphi_n(x) - \psi_m(y) = c(x,y).
\end{align*}
Thus, all in all we proved that $(\varphi, \psi)$ is feasible. It remains to prove that $(\varphi, \psi)$ is complementary to $\gamma$. By Theorem~\ref{thm:Components}-(i) we have that $\gamma(X_n \times Y_m)=0$ for all $n,m \in \{1, \ldots, N\}$ with $n \neq m$. Hence, we only have to check that 
\[
\gamma(x,y)>0 \Rightarrow c(x,y) = \varphi(x)+\psi(y) \quad \text{for all } n \in \{1, \ldots, N\}, x \in X_n, y \in Y_n.
\] 
By definition of $\varphi$ and $\psi$, this is equivalent to showing that
\[
\gamma(x,y)>0 \Rightarrow c(x,y) = \varphi_n(x)+\psi_n(y) \quad \text{for all } n \in \{1, \ldots, N\}, x \in X_n, y \in Y_n.
\] 
This follows from the fact that $\frac{1}{\mu(X_n)} \gamma|_{X_n \times Y_n}$ is optimal for $\OT_{X_n,Y_n}(\mu, \nu, c)$ by Theorem~\ref{thm:Components}-(ii),
and $(\varphi_n, \psi_n)$ is optimal for $\DOT_{X_n, Y_n} (\mu, \nu, c)$ by assumption, thus they are complementary, which yields
\[
\gamma(x,y)>0 \Leftrightarrow \frac{1}{\mu(X_n)} \gamma|_{X_n \times Y_n} (x,y)>0\Rightarrow c(x,y) = \varphi_n(x)+\psi_n(y) \;\, \text{for all } n \in \{1, \ldots, N\}, x \in X_n, y \in Y_n.
\]
\end{proof}

\begin{example}
\label{ex2}
Consider $\Xcal = \{1,2,3,4\}$ and $\Ycal = \{1,2,3,4,5\}$ and the following measures: 
\[
\mu = \tfrac{3}{20} \delta_{\{1\}} + \tfrac{3}{20} \delta_{\{2\}} + \tfrac{1}{5} \delta_{\{3\}} + \tfrac{1}{2} \delta_{\{4\}} \quad \text{and} \quad \nu = \tfrac{1}{5} \delta_{\{1\}} + \tfrac{1}{10} \delta_{\{2\}} + \tfrac{1}{5} \delta_{\{3\}} + \tfrac{1}{4} \delta_{\{4\}} + \tfrac{1}{4} \delta_{\{5\}}.
\]
Set 
$F_1 =\{ (2,1), (3,3), (3,4)\},\,
    F_2 =\{ (1,1), (2,2), (4,3), (4,4), (4,5)\},\,
    F_3 = \{(1,2)\}$,
and define
\[
c(x,y) = 1_{F_2}(x,y)+3\cdot 1_{F_3}(x,y)+2\cdot 1_{(F_1\cup F_2\cup F_3)^c}(x,y).
\]
Then one finds that 
\[
\gamma= \tfrac{3}{20}\cdot 1_{\{(1,1)\}}+
\tfrac{1}{20}\cdot 1_{\{(2,1)\}}+
\tfrac{1}{10}\cdot 1_{\{(2,2)\}}+
\tfrac{1}{5}\cdot 1_{\{(3,3)\}}+
\tfrac{1}{4}\cdot 1_{\{(4,4), (4,5)\}}
\]
is a primal optimizer for $\OT(\mu, \nu, c)$. The associated graph $G^\gamma$ is depicted in Figure~\ref{fig:ex2} and has three components with 
$X_1 = \{1,2\}, Y_1 = \{1,2\}, X_2 = \{3\}, Y_2 = \{3\}, X_3= \{4\}, Y_3 = \{4,5\}$.
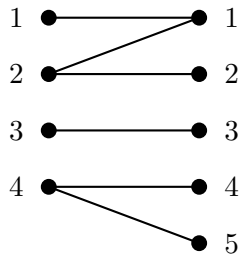
\begin{figure}[h]
\centering
{\begin{tikzpicture}
\clip (-0.75,0) rectangle (2.75, 3.75);

\coordinate[label={[label distance=0.2cm]left:$4$}] (L4) at (0,1.25);
\fill (L4) circle (3pt);
\coordinate[label={[label distance=0.2cm]left:$3$}] (L3) at (0,2);
\fill (L3) circle (3pt);
\coordinate[label={[label distance=0.2cm]left:$2$}] (L2) at (0,2.75);
\fill (L2) circle (3pt);
\coordinate[label={[label distance=0.2cm]left:$1$}] (L1) at (0,3.5);
\fill (L1) circle (3pt);
\coordinate[label={[label distance=0.2cm]right:$5$}] (R5) at (2,0.5);
\fill (R5) circle (3pt);
\coordinate[label={[label distance=0.2cm]right:$4$}] (R4) at (2,1.25);
\fill (R4) circle (3pt);
\coordinate[label={[label distance=0.2cm]right:$3$}] (R3) at (2,2);
\fill (R3) circle (3pt);
\coordinate[label={[label distance=0.2cm]right:$2$}] (R2) at (2,2.75);
\fill (R2) circle (3pt);
\coordinate[label={[label distance=0.2cm]right:$1$}] (R1) at (2,3.5);
\fill (R1) circle (3pt);
\draw[thick] (L1) -- (R1) -- (L2) -- (R2);
\draw[thick] (R3) -- (L3);	
\draw[thick] (R4) -- (L4) -- (R5);

\end{tikzpicture}}

\caption{The graph $G^\gamma$ from Example~\ref{ex2} and Example~\ref{ex3} %and the graph $T$ from Example~\ref{ex3}
}
\label{fig:ex2}
\end{figure}
The dual optimizers for the connected components can be immediately computed as explained after Proposition~\ref{prop:unique}, and they read as
\begin{align*}
&\varphi_1(1)=0, \, \varphi_1(2)= -1, \, \psi_1(1)=1, \, \psi_1(2)=2 \\
&\varphi_2(3)=0, \, \psi_2(3)=0 \\
&\varphi_3(4)=0, \, \psi_3(4)=1, \, \psi_3(5)=1.
\end{align*}
Now the constraints \eqref{eq.anm} on $\alpha=(\alpha_1, \alpha_2, \alpha_3)$ are
\[
0 \le \alpha_1 - \alpha_2 \le 2, \; 0 \le  \alpha_1 - \alpha_3 \le 1 \text{ and } -1 \le \alpha_2 - \alpha_3 \le -1.
\] 
Hence, we obtain that any suitable $\alpha$ has to satisfy $\alpha_1 \in \mathbb{R}$, $\alpha_2 \in [\alpha_1 - 2, \alpha_1]$ and $\alpha_3 = \alpha_2 + 1$. By Theorem~\ref{thm:SetOfAllDuals}, these choices describe all dual optimizers. \hfill$\diamond$
\end{example}

Given this result, we can formulate the announced statement regarding the relation of $E$  and the set $\{[x,y]: x\in\Xcal, y\in\Ycal, c(x,y) = \varphi(x) + \psi(y)\}$ for not necessarily connected graphs.

\begin{corollary}
\label{cor:relationEandPhi}
Let $G$ be the graph associated to $\OT(\mu,\nu,c)$ and $V_1, \ldots, V_N$ its connected components, 
with $V_n = X_n \cup Y_n$ for subsets $X_n \subseteq \Xcal$ and $Y_n\subseteq \Ycal$ for  $n \in \{1, \ldots, N\}$. Let $(\varphi_n,\psi_n)$ be the unique optimizer for $\DOT_{X_n,Y_n}(\mu, \nu, c)$ for $n \in \{1, \ldots, N\}$. Moreover, let $(\varphi, \psi)$ be an optimizer for $\DOT(\mu,\nu,c)$, and $\alpha =(\alpha_1, \ldots, \alpha_N)$ be the constants satisfying condition (i)-(iii) of Theorem~\ref{thm:SetOfAllDuals}. Then
\[
E = \{[x,y]: x\in\Xcal, y\in\Ycal, c(x,y) = \varphi(x) + \psi(y)\}
\] 
if and only if, for all $n, m \in \{1, \ldots, N\}$ with $n \neq m$, \eqref{eq.anm} holds with strict inequalities.
\end{corollary}

\begin{proof}
By Corollary~\ref{cor:Components}-(ii), the graph $G[X_n \cup Y_n]$ is the graph associated to the subproblem $\OT_{X_n, Y_n}(\mu, \nu, c)$.
Moreover, $G[X_n \cup Y_n]$  is connected and $(\varphi_n, \psi_n )$ is the dual optimizer for $\OT_{X_n, Y_n}(\mu, \nu, c)$.
Hence, by Proposition~\ref{prop:ChracterizationGconnected}, we obtain for $x \in X_n$ and $y \in Y_n$ that $(x,y) \in E$ if and only if 
\[
c(x,y) = \varphi_n(x) + \psi_n(y) = \varphi_n(x) + \alpha_n + \psi_n(y) - \alpha_n = \varphi(x) + \psi(y).
\] 
Now let $n, m \in \{1, \ldots, N\}$ with $n \neq m$. If $\alpha_n-\alpha_m$ satisfies \eqref{eq.anm} with strict inequalities, then for any $x \in X_n$ and $y \in Y_m$ we have that $(x,y) \notin E$ by Corollary~\ref{cor:Components}-(i), and moreover
\[
\varphi(x) + \psi(y) = \varphi_n(x) + \alpha_n + \psi_m(y) - \alpha_m < \varphi_n(x) + \psi_m(y) + c(x,y) - \varphi_n(x) - \psi_m(y) = c(x,y).
\]  
Analogously, we obtain for $x \in X_m$ and $y \in Y_n$ that $(x,y) \notin E$ and $c(x,y) > \varphi(x) + \psi(y)$. Hence, for any dual optimizer with $\alpha$ satisfying \eqref{eq.anm} with strict inequalities for all $n,m \in \{1, \ldots, N\}$ with $n \neq m$, we have that $(x,y) \in E$ if and only if $c(x,y) = \varphi(x)+\psi(y)$.

Now assume that, for some $n,m \in \{1, \ldots, N\}$ with $n \neq m$, $\alpha_n-\alpha_m$ equals one of the extremes in \eqref{eq.anm}. Assume first that $\alpha_n - \alpha_m = - \min_{x \in X_m, y \in Y_n} c(x,y) -\varphi_m(x)- \psi_n(y)$ and let $x_m \in X_m$ and $y_n \in Y_n$ be such that $\alpha_n - \alpha_m = - (c(x_m,y_n) -\varphi_m(x_m)- \psi_n(y_n))$. Then, again by Corollary~\ref{cor:Components}-(i), we have that $(x_m, y_n) \notin E$, and at the same time
\[
c(x_m,y_n) =  \varphi_m(x_m) + \alpha_m + \psi_n(y_n) - \alpha_n = \varphi(x_m) + \psi(y_n).
\] 
Analogously for the other extreme. This concludes the proof. 
\end{proof}

We know that whenever the interval
\[
\left[ - \min_{x \in X_m, y \in Y_n} c(x,y) -\varphi_m(x)- \psi_n(y), \min_{x \in X_n, y \in Y_m} c(x,y) - \varphi_n(x) -\psi_m(y) \right]
\] 
consists of more than one value, then there are multiple dual optimizers and hence $G[X_n \cup Y_n \cup X_m \cup Y_m]$ cannot be connected, by Proposition~\ref{prop:unique}-(ii). In the next lemma we show that if the interval consists of exactly one point, then $G[X_n \cup Y_n \cup X_m \cup Y_m]$ is indeed connected.

\begin{prop}
\label{prop:MaximalComponents}
In the setting of Theorem~\ref{thm:SetOfAllDuals}, if for some $n,m \in \{1, \ldots, N\}$ with $n\neq m$ we have
\begin{equation}\label{eq.mm}
\min_{x \in X_n, y \in Y_m} c(x,y) - \varphi_n(x) -\psi_m(y) = - \min_{x \in X_m, y \in Y_n} c(x,y) -\varphi_m(x)- \psi_n(y) 
\end{equation}
then $G[X_n\cup Y_n \cup X_m \cup Y_m]$ is connected.
\end{prop}

\begin{proof}
Let $n,m \in \{1, \ldots, N\}$, $n\neq m$ be such that the equality in \eqref{eq.mm} holds, and call $\beta$ the common value of the LHS and RHS of \eqref{eq.mm}.
Let $(\phi, \psi)$ be a dual optimizer. Then by Theorem~\ref{thm:SetOfAllDuals} there are constants $\alpha_1, \alpha_2, \ldots, \alpha_N$ such that conditions (i) - (iii) of Theorem~\ref{thm:SetOfAllDuals} are met. In particular, $\alpha_n - \alpha_m = \beta$. Since $(\varphi, \psi)$
is, as a dual optimizer, feasible, we have
\begin{align*}
0 &\le \min_{x \in X_n, y \in Y_m} c(x,y) - \varphi(x) - \psi(y) = \min_{x \in X_n, y \in Y_m} c(x,y) - \varphi_n(x) - \alpha_n  - \psi_m(y) + \alpha_m\\
&= \min_{x \in X_n, y \in Y_m} c(x,y) - \varphi_n(x) - \psi_m(y) - \beta = 0. 
\end{align*} 
Together with \eqref{eq.mm}, this implies that there are $x_n \in X_n$, $x_m \in X_m$, $y_n \in Y_n$ and $x_m \in X_m$ such that 
\begin{align}
\label{eq:Comp_Binding}
c(x_n,y_m) = \varphi(x_n) + \psi(y_m) \quad \text{and} \quad c(x_m,y_n) = \varphi(x_m) + \psi(y_n).
\end{align} 

Let $i \in \{n,m\}$. Since $G[X_i \cup Y_i]$ is connected and bipartite, there is a path $x_i^{(0)}y_i^{(0)}x_i^{(1)}\ldots x_i^{(l_i)} y_i^{(l_i)}$ from $x_i^{(0)}= x_i$ to $y_i^{(l_n)} = y_i$ in $X_i \cup Y_i$,  with $x_i^{(j)} \in X_i$ and $y_i^{(j)} \in Y_i$ for all $j \in \{0, \ldots, l_i\}$. 
Note that there is no vertex that appears in both paths since $X_n \cup Y_n$ and $X_m \cup Y_m$ are disjoint.

Now set 
\begin{equation*}
\delta =  \min_{i\in\{n,m\}, j \in \{0, \ldots, l_i\}} \gamma( x_i^{(j)}, y_i^{(j)}) \wedge \min_{i\in\{n,m\}, j \in \{1, \ldots, l_{i}\}} (1- \gamma(x_i^{(j)}, y_i^{(j-1)})),
\end{equation*} 
and define $\hat{\gamma}: \Xcal \times \Ycal \rightarrow [0,1]$ by
\[
\hat{\gamma}(x,y) =
\begin{cases}
    \gamma(x,y) - \delta &\text{if } (x,y) = (x_i^{(j)}, y_i^{(j)}) \text{ for some } i\in\{n,m\}, j \in \{0, \ldots, l_i\}  \\
    \gamma(x,y) + \delta &\text{if } (x,y) = (x_i^{(j)}, y_i^{(j-1)}) \text{ for some } i\in\{n,m\}, j \in \{1, \ldots, l_i\}  \\
    \gamma(x,y) + \delta &\text{if } (x,y) \in \{(x_n, y_m), (x_m, y_n)\} \\
    \gamma(x,y) &\text{else.}
\end{cases}
\] 
The construction of $\hat{\gamma}$ is illustrated in Figure~\ref{fig:proofIllustration}.

\begin{figure}[h]
\centering
\begin{tikzpicture}
\coordinate[label={[label distance=0.2cm]left:$x_n^{(0)}$}] (AL1) at (0,2);
\fill (AL1) circle (3pt);
\coordinate[label={[label distance=0.2cm]right:$y_n^{(l_n)}$}] (AR1) at (3,2);
\fill (AR1) circle (3pt);
\coordinate[label={[label distance=0.2cm]left:$y_n^{(0)}$}] (AL2) at (0.25,2.75);
\fill (AL2) circle (3pt);
\coordinate[label={[label distance=0.2cm]right:$x_n^{(l_n)}$}] (AR2) at (2.75,2.75);
\fill (AR2) circle (3pt);
\coordinate[label={[label distance=0.2cm]left:$x_n^{(1)}$}] (AL3) at (0.75,3.5);
\fill (AL3) circle (3pt);
\coordinate[label={[label distance=0.2cm]right:$y_n^{(l_n-1)}$}] (AR3) at (2.25,3.5);
\fill (AR3) circle (3pt);
\coordinate[label={[label distance=0cm]below:$\cdots$}] (CD) at (1.5,3.65);

\draw[rounded corners = 0.5cm] (-1.25, 1.65) -- (4.25, 1.65) -- (4, 3.8) -- (-0.75, 3.8) -- cycle;
\coordinate[label={[label distance=0cm]above:$V_n$}] (CD) at (1.5,3.8);

\draw[thick] (AL1) -- (AL2) node [midway, right] {$- \delta$} ;
\draw[thick, dashed] (AL2) -- (AL3) node [midway, right] {$+ \delta$} ;
\draw[thick] (AR2) -- (AR1) node [midway, left] {$- \delta$} ;
\draw[thick, dashed] (AR3) -- (AR2) node [midway, left] {$+ \delta$} ;

\coordinate[label={[label distance=0.2cm]left:$x_m^{(0)}$}] (BL1) at (0,0.5);
\fill (BL1) circle (3pt);
\coordinate[label={[label distance=0.2cm]right:$y_m^{(l_m)}$}] (BR1) at (3,0.5);
\fill (BR1) circle (3pt);
\coordinate[label={[label distance=0.2cm]left:$y_m^{(0)}$}] (BL2) at (0.25,-0.25);
\fill (BL2) circle (3pt);
\coordinate[label={[label distance=0.2cm]right:$x_m^{(l_m)}$}] (BR2) at (2.75,-0.25);
\fill (BR2) circle (3pt);
\coordinate[label={[label distance=0.2cm]left:$x_m^{(1)}$}] (BL3) at (0.75,-1);
\fill (BL3) circle (3pt);
\coordinate[label={[label distance=0.2cm]right:$y_m^{(l_m-1)}$}] (BR3) at (2.25,-1);
\fill (BR3) circle (3pt);
\coordinate[label={[label distance=0cm]above:$\ldots$}] (CD) at (1.5,-1.15);

\draw[rounded corners = 0.5cm] (-1.25, 0.85) -- (4.25, 0.85) -- (4, -1.45) -- (-0.75, -1.45) -- cycle;
\coordinate[label={[label distance=0cm]below:$V_m$}] (CD) at (1.5,-1.6);

\draw[thick] (BL1) -- (BL2) node [midway, right] {$- \delta$} ;
\draw[thick, dashed] (BL2) -- (BL3) node [midway, right] {$+ \delta$} ;
\draw[thick] (BR1) -- (BR2) node [midway, left] {$- \delta$} ;
\draw[thick, dashed] (BR3) -- (BR2) node [midway, left] {$+ \delta$} ;
\draw[thick, dashed] (BL1) -- (AR1) node [pos=0.37, above left, sloped] {$+ \delta$} ;
\draw[thick, dashed] (BR1) -- (AL1) node [pos=0.37, above right, sloped] {$+ \delta$} ;

\end{tikzpicture}
\caption{The figure illustrates the construction of $\hat{\gamma}$. This shows on which edges $[x,y]$ we set $\hat{\gamma}(x,y) = \gamma(x,y) - \delta$ (solid edges) and on which edges $[x,y]$ we set $\hat{\gamma}(x,y) = \gamma(x,y) - \delta$ (dashed edges).}
\label{fig:proofIllustration}
\end{figure}
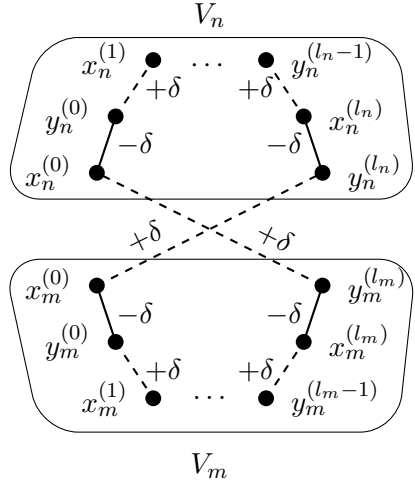

We are going to show that $\hat{\gamma} \in \Pi(\mu,\nu)$. Since $\gamma(x_n,y_m)=\gamma(x_m,y_n)=0$, by definition of $\delta$ we have that $\gamma(x,y) \in [0,1]$ for all $(x,y) \in \Xcal \times \Ycal$. Moreover, for any $x \in \Xcal \setminus \{x_n^{(0)}, \ldots, x_n^{(l_n)}, x_m^{(0)}, x_m^{(l_m)}\}$, we have
\[
\sum_{y \in \Ycal} \hat{\gamma}(x,y) = \sum_{y \in \Ycal} \gamma(x,y) =\mu(x).
\] 
On the other hand, for $i\in\{n,m\}$ and $j \in \{0,\ldots, l_n\}$, we obtain
\begin{align*}
\sum_{y \in \Ycal} \hat{\gamma}(x_i^{(j)},y) &= \sum_{y \in \Ycal \setminus \{y_i^{(j)}, y_i^{(j-1)}\}} \gamma(x_i^{(j)},y) + \gamma(x_i^{(j)},y_i^{(j)}) - \delta + \gamma(x_i^{(j)}, y_i^{(j-1)}) + \delta \\
&= \sum_{y \in \Ycal} \gamma(x_i^{(j)},y) = \mu(x_i^{(j)}),
\end{align*} 
where we used the convention $y_i^{(-1)}=y_i^{(l_i)}$. 
Analogous computations for $y \in \Ycal$ show that $\hat{\gamma} \in \Pi(\mu, \nu).$

Now we prove that $\hat{\gamma}$ and $(\varphi, \psi)$ are complementary.
We start by noticing that for $(x,y) \notin \{(x_n,y_m), (x_m,y_n) \}$ we have 
\[
\hat{\gamma}(x,y) >0 \Rightarrow \gamma(x,y)>0.
\] 
Moreover, for $(x,y) \in \{(x_n, y_m), (x_m,y_n)\}$ we have \eqref{eq:Comp_Binding}. Hence, the complementarity of $\gamma$ and $(\varphi, \psi)$ implies that $\hat{\gamma}$ and $(\varphi, \psi)$ are complementary. By Theorem~\ref{thm.510}-(iv) this implies that the coupling $\hat{\gamma}$ is a primal optimizer. Thus, $[x_n,y_m], [x_m,y_n] \in E$, which yields that $G[X_n \cup Y_n \cup X_m \cup Y_m]$ is connected, as wanted.
\end{proof}

\begin{continueexample}{ex2}
We have seen that, for any admissible $\alpha$, the intervals $\alpha_1 - \alpha_2$ and $\alpha_1 - \alpha_3$ consist of more than one point, which means that $X_1 \cup Y_1$ and $X_2 \cup Y_2$ as well as $X_1 \cup Y_1$ and $X_3 \cup Y_3$ are not part of the same component of $G$. On the other hand, the interval $\alpha_2 - \alpha_3$ consists of exactly one point. Hence, in $G$ there are two components with $\tilde{X}_1 = \{1,2\}$, $\tilde{Y}_1 = \{1,2\}$, $\tilde{X}_2 = \{3,4\}$ and $\tilde{Y}_2 = \{3,4,5\}$.
\hfill$\diamond$
\end{continueexample}

\begin{remark} The results derived so far allow us to compute $G$ given one primal optimizer $\gamma$ and one dual optimizer $(\varphi, \psi)$. Namely, given $G^\gamma$, we first determine the components $V_1, \ldots, V_N$ of $G^\gamma$ and then check for which pairs $n,m \in \{1, \ldots, N\}$ condition \eqref{eq.mm} is satisfied. If this condition is satisfied, by Proposition~\ref{prop:MaximalComponents} we have that $V_n \cup V_m$ is part of one component of $G$. 
If this condition is not satisfied, then as noted before Proposition~\ref{prop:MaximalComponents} we have that $V_n$ and $V_m$ cannot be part of the same component. Hence, by checking the condition for all $n,m \in \{1, \ldots, N\}$, we obtain the connected components $\Tilde{V}_1, \ldots, \tilde{V}_M$ of $G$.
Then we construct a subgraph $G' \subseteq G$ with the same components as $G$ in the following way. For each pair $(n,m)$ such that  \eqref{eq.mm} is satisfied, we find pairs $(x_n,y_m)\in X_n\times Y_m$ and $(x_m,y_n)\in X_m \times Y_n$ that attain the minima in \eqref{eq.mm}. As argued in the proof of Proposition~\ref{prop:MaximalComponents}, we have that $(x_n,y_m), (x_m,y_n) \in E$. Now let $G'$ be the graph obtained by adding to $G^\gamma$ the edges $(x_n,y_m), (x_m,y_n)$ for any $(n,m)$ such that \eqref{eq.mm} holds. By construction, this graph has the same connected components as $G$. Finally, we compute the edges of $G$. By Proposition~\ref{prop:unique}-(ii), the dual optimizer of each component $\tilde{V}_n=\Tilde{X}_n\cup\Tilde{Y}_n$ is unique, therefore $(\varphi, \psi)$ restricted to $\Tilde{V}_n$ is the unique optimizer. Hence, by Proposition~\ref{prop:ChracterizationGconnected}, all edges of $G$ in this component are given by those edges $(x,y) \in \Tilde{X}_n \times \Tilde{Y}_n$ for which $c(x,y)=\varphi(x) + \psi(y)$. That these are all edges of $G$ finally follows from Corollary~\ref{cor:relationEandPhi}.
\hfill$\diamond$
\end{remark}

Finally, we can now show that the condition that $G$ is connected is also necessary for the uniqueness (up to translation) of the dual optimizer.

\begin{corollary}
    \label{cor:UniqueMeansConnected}
    The dual optimizer is unique if and only if the graph $G$ is connected.
\end{corollary}

\begin{proof}
    That $G$ being connected is sufficient for a unique dual optimizer has been proved in Proposition~\ref{prop:unique}. Hence, it remains to prove that the condition is also necessary for uniqueness.    
    Assume that there is a unique dual optimizer and, by contradiction, that the graph $G$ is not connected, i.e. that $G$ has $N \ge 2$ components. By Theorem~\ref{thm:SetOfAllDuals} we now have that for any $\alpha =(\alpha_1, \ldots, \alpha_N)$ satisfying condition \eqref{eq.anm} there is a dual optimizer. Since, by assumption, the dual optimizer is unique, $\alpha$ has to be uniquely determined, which in particular means that 
    \[
    \alpha_1 - \alpha_2 = \min_{x \in X_1, y \in Y_2} c(x,y) - \varphi_1(x) - \psi_2(y) = -\min_{x \in X_2, y \in Y_1} c(x,y) - \varphi_2(x) - \psi_1(y).
    \] By Proposition~\ref{prop:MaximalComponents} this implies that $G[X_1 \cup Y_1 \cup X_2 \cup Y_2]$ is connected, which is the desired contradiction. 
\end{proof}

\section{Entropic Transport Problems and Asymptotics.}
\label{sec:entropic}
In this section we turn to the entropic regularization of the optimal transport problem and provide the characterization of the limit of the dual entropic optimizers. We start by briefly describing the entropic regularization of the optimal transport problem. For more details in this discrete setting, we refer the reader to Chapter 4 in \cite{ComputationalOT}.

For any $\eps>0$,  the entropic regularization of $\OT(\mu, \nu, c)$ reads as
\begin{equation}\label{eq.eOT}
\OT_\eps(\mu, \nu, c) = \inf_{\gamma \in \Pi(\mu, \nu)} \left\{\textstyle{\int c\ d\gamma + \eps H(\gamma)}\right\},
\end{equation}
where $H(\gamma)$ is the entropy associated to the coupling $\gamma$, and is given by
\[
H(\gamma)=\sum_{(x,y) \in \Xcal \times \Ycal} \gamma(x,y) (\ln (\gamma(x,y)) -1),
\]
with the usual convention $0\log(0)=0$.
The corresponding dual is
\begin{align}\label{eq.DeOT}
\DOT_\eps (\mu, \nu, c) = \sup \Big\{\textstyle{ \int_\Xcal \varphi d\mu + \int_\Ycal\psi d\nu} - \eps \displaystyle{\sum_{(x,y) \in \Xcal \times \Ycal}} \text{e}^{\frac{1}{\eps} (-c(x,y)+\varphi(x)+\psi(y))}
: \varphi:\Xcal\to\RR, \psi:\Ycal\to\RR\Big\}.
\end{align}
As for the classical transport problem, duality $\OT_\eps (\mu, \nu, c)=\DOT_\eps (\mu, \nu, c)$ holds. One main advantage of the regularized problems is that both \eqref{eq.eOT} and \eqref{eq.DeOT} admit unique optimizers, that we call primal and dual entropic optimizers, respectively, and denote them by $\gamma^\eps$ and $(\varphi^\eps, \psi^\eps)$.

\begin{prop}[\cite{ComputationalOT}, Proposition 4.1]
For $\eps \rightarrow 0$, the unique solution $\gamma^\eps$ of $\OT_\eps(\mu, \nu, c)$
converges to a solution of $\OT(\mu, \nu, c)$, and specifically to the one with minimal entropy, namely
\begin{equation}\label{eq.gamma*}
\gamma^\eps\to\gamma^*
:=\argmin \left\{H(\gamma) : \gamma \text{ optimizer of } \eqref{eq.OT} \right\}.
\end{equation}
As a consequence, $\OT_\eps(\mu, \nu, c)\to\OT(\mu, \nu, c)$ for $\eps \rightarrow 0$.
\end{prop}
Note that the minimizer $\gamma^*$ in \eqref{eq.gamma*} is unique because of strict convexity of relative entropy and convexity of the set of primal optimizers.

\begin{remark}
By Remark~\ref{rem.supp}, we obtain the more convenient characterization 
\[
\gamma^*=\argmin \left\{\sum_{[x,y] \in E} \gamma(x,y) (\ln(\gamma(x,y))-1) : \gamma \in \Pi(\mu, \nu) \text{ s.t. supp$(\gamma)\subseteq E$}\right\}.
\]
\hfill $\diamond$
\end{remark}

Clearly, the discrete optimal transport problem $\OT(\mu,\nu,c)$ is a linear programming problem. 
Moreover, by fixing an arbitrary $x_0 \in \Xcal$ and dropping the redundant constraint $\sum_{y \in \Ycal} \gamma(x_0,y)=\mu(x_0)$, we obtain a linear programming problem where the matrix describing the constraints has full rank; see \cite[Remark 3.1]{ComputationalOT}. The dual problem associated to this adjusted linear programming problem is problem $\text{DOT}(\mu, \nu, c)$ with the additional requirement that $\varphi(x_0)=0$. By Theorem~\ref{thm.510} there is a dual optimizer and by Theorem~\ref{thm:SetOfAllDuals} we moreover have that the set of all solutions to this problem is bounded. Hence, the adjusted dual problem satisfies the assumptions requested in
\cite[Proposition~3.2]{CominettiSanMartin}, ensuring the existence of a dual optimizer $(\hat\varphi, \hat \psi)$, namely the centroid of the solution set, such that  any sequence of dual entropic optimizers $(\varphi^{\eps_n}, \psi^{\eps_n})$ with $\epsilon_n \rightarrow 0$ converges to $(\hat\varphi, \hat \psi)$. However, the construction of the centroid in \cite{CominettiSanMartin} requires the determination of the solution set for up to $n_\Xcal \times n_\Ycal$ convex optimization problems. We now propose a simple algorithm to pinpoint $(\hat\varphi, \hat \psi)$, that relies only on elementary computations.

In what follows we denote by $V_1, \ldots, V_N$ the connected components of the graph $G$ associated to $\OT(\mu,\nu,c)$, and we consider the usual decomposition $V_n=X_n \cup Y_n$ for $n \in \{1, \ldots, N\}$. Without loss of generality we assume that $x_0 \in X_1$. Recall that, for each $n$, the dual problem $\DOT_{X_n,Y_n}(\mu,\nu,c)$ admits a solution that is unique up to translation by constants.
We fix one of them and denote it by $(\varphi_n,\psi_n)$. For the first component, we choose this representative so that $\varphi_1(x_0)=0$.

\begin{figure}[ht]
\begin{algorithm}[H]
\DontPrintSemicolon
$F \longleftarrow [\{1, \ldots, N\}]^2$ \;
$E^T \longleftarrow \emptyset$ \;
\ForEach{$(n,m) \in F$}{
$\delta_{n,m} \longleftarrow \frac{1}{2} \left(\min_{x \in X_n, y \in Y_m} c(x,y) - \varphi_n(x) - \psi_m(y) + \min_{x \in X_m, y \in Y_n} c(x,y) - \varphi_m(x)- \psi_n(y)\right)$
}
\While{$F$ non-empty}{

$\delta \longleftarrow \min_{(n,m) \in F} \delta_{n,m}$\;
\For{$(n,m)\in F: \delta_{n,m} = \delta$}{
$E^T \longleftarrow E^T \cup \{(n,m)\}$ \; 
remove $(n',m')$ from $F$ whenever there is a path from $n'$ to $m'$ in $T$\;			
}
}
\caption{Construction of the tree $T$}
\label{Algorithm:ConstructionT}
\end{algorithm}
\caption{By means of simple computations, Algorithm~\ref{Algorithm:ConstructionT} gives a recipe to construct a tree $T$ on the vertex set $\{1, \ldots, N\}$, which is a connected graph with exactly $N-1$ edges.
This tree then allows us to describe the dual optimizer to which any sequence of dual entropic optimizers converges, as described in Theorem~\ref{thm:LimitEntropic}.}
\end{figure}

\begin{theorem}
\label{thm:LimitEntropic}
Let $T=(V^T,E^T)$ be the spanning tree on $V^T=\{1,...,N\}$ given by Algorithm~\ref{Algorithm:ConstructionT}.
Let $\alpha_1=0$ and set $\alpha_2,\ldots,\alpha_N$ to be constants such that, for all $\{n,m\} \in E^T$,
\begin{equation}\label{eq:DualLimitcEqual}
\alpha_n - \alpha_m = L_{n,m}:= \frac{1}{2}  \min_{x \in X_n, y \in Y_m} \big(c(x,y) - \varphi_n(x) - \psi_m(y) \big) - \frac{1}{2} \min_{x \in X_m, y \in Y_n} \big(c(x,y) - \varphi_m(x) - \psi_n(y) \big).
\end{equation}
Define $\varphi^\ast:\Xcal\to\RR$ and $\psi^\ast:\Ycal\to\RR$ by $\varphi^\ast(x) = \varphi_n(x) + \alpha_n$ and $\psi^\ast (y) = \psi_n(y) - \alpha_n$ for all $x \in X_n$ and $y \in Y_n$, $n\in\{1,\ldots,N\}$.
Then any sequence $(\varphi^{\eps_n}, \psi^{\eps_n})$ with $\eps_n \rightarrow 0$ converges to $(\varphi^\ast, \psi^\ast)$.	
\end{theorem}

By what said above, proving Theorem~\ref{thm:LimitEntropic} corresponds to showing that $(\varphi^\ast, \psi^\ast)$ is the centroid $(\hat\varphi, \hat \psi)$ found in \cite{CominettiSanMartin}.

\begin{proof}
\emph{Step 1: There is exactly one set of constants $(\alpha_1, \ldots, \alpha_N)$, with $\alpha_1=0$, that satisfies \eqref{eq:DualLimitcEqual} for all $\{n,m\} \in E^T$.}\\
Note that the graph $T$ built in Algorithm~\ref{Algorithm:ConstructionT} is a tree, which in particular means that it is connected. 
Then, by Lemma~\ref{lemma.ordering} we find an ordering $v_1, \ldots, v_N$ of the vertices in $T$ such that $v_1=1$ and $T[v_1, \ldots, v_i]$ is connected for every $i \in \{2, \ldots, N\}$. As in the proof of Proposition~\ref{prop:unique}, we can then successively set $\alpha_{v_i}$ for any $i \in \{2, \ldots, N\}$. Since $T$ is a tree, for any $i \in \{2, \ldots, N\}$ there is exactly one edge $\{v_{i'}, v_i\}$ with $i'<i$ in $T$ and we choose $\alpha_{v_i}$ such that $\alpha_{v_i} - \alpha_{v_{i'}}$ satisfies \eqref{eq:DualLimitcEqual}.
Proceeding this way, we uniquely determine $(\alpha_2, \ldots, \alpha_n)$. Moreover, we used one edge for each step $i \in \{2, \ldots, N\}$, and all these edges are distinct. Hence, we used all $N-1$ edges in $T$. Therefore, $(\alpha_1, \ldots, \alpha_N)$ satisfies \eqref{eq:DualLimitcEqual} for all $\{n,m\} \in E^T$.

\emph{Step 2: Description of the construction procedure for the centroid in Cominetti and San Mart\'{i}n~\cite{CominettiSanMartin}.}\\
We denote by $S_0$ the set of all solutions of the dual problem $\DOT(\mu,\nu,c)$. Let $I_0 = \{(x,y) \in \Xcal \times\Ycal: c(x,y) = \varphi(x)+\psi(y) \; \forall (\varphi, \psi) \in S_0\}$. By Proposition~\ref{prop:Gcharacterizes} we have that $E\subseteq I_0$. Since, by Corollary~\ref{cor:relationEandPhi}, there exists a dual optimizer $(\varphi, \psi)$ such that $E = \{ (x,y) \in \Xcal \times \Ycal: c(x,y) = \varphi(x) + \psi(y)\}$, we obtain that $E=I_0$. 
Next, for each $n=0,...,n_\Xcal \cdot n_\Ycal-1$, we define the continuous concave function
\[
f_{n}(\varphi, \psi) = \min_{(x,y) \notin I_{n}} c(x,y) - \varphi(x) - \psi(y).
\] 
For $n=1,...,n_\Xcal \cdot n_\Ycal$, we consider the convex optimization problem
\[
w_n=\max \{f_{n-1}(\varphi, \psi): (\varphi, \psi) \in S_{n-1}\},
\] 
and we denote by $S_n$ the set of its solutions. 
%Indeed, we are maximizing a continuous concave function over a bounded closed convex set.
Finally, we write 
\[
J_n = \{ (x,y) \notin  I_{n-1}: c(x,y) - w_n = \varphi(x)+\psi(y)\ \forall (\varphi, \psi) \in S_n\}
\]
and set $I_n = I_{n-1} \cup J_n$.  
Note that we start this induction procedure with the set $S_0$ which is non-empty, bounded, closed and convex, and that by construction all sets $S_n$ preserve the same properties.

\emph{Step 3: The set $J_n$ is non-empty for all $n$ such that $I_{n-1}\subsetneq \Xcal\times \Ycal$.}\\ 
We work by way of contradiction, and assume that $J_n$ is empty. This means that for any $(x,y) \notin I_{n-1}$ there is a pair $(\varphi^{(x,y)}, \psi^{(x,y)}) \in S_n$ such that
\begin{equation}
\label{eq:Jn-non-empty}
c(x,y) - \varphi^{(x,y)}(x) - \psi^{(x,y)}(y) > w_n.
\end{equation}
Since $S_n$ is convex, we have that
\[
(\bar\varphi, \bar\psi) 
:= \Big( \tfrac{1}{n_\Xcal \times n_\Ycal - |I_{n-1}|} \sum_{(x,y)\notin I_{n-1}} \varphi^{(x,y)} , \tfrac{1}{n_\Xcal \times n_\Ycal - |I_{n-1}|} \sum_{(x,y)\notin I_{n-1}} \psi^{(x,y)} \Big) \in S_n.
\]
Let $(x',y') \notin I_{n-1}$ be arbitrary. Then we have  
\[    c(x',y') - \bar\varphi(x') - \bar\psi(y') = \tfrac{1}{n_\Xcal \times n_\Ycal - |I_{n-1}|} \sum_{(x,y)\notin I_{n-1}} \left( c(x',y')-\varphi^{(x,y)}(x') - \psi^{(x,y)}(y') \right).
\] 
Note that $c(x',y')-\varphi^{(x,y)}(x') - \psi^{(x,y)}(y') \ge w_n$ for all $(x,y) \notin I_{n-1}$, since $(\varphi^{(x,y)}, \psi^{(x,y)}) \in S_n$. Moreover,  $c(x',y')-\varphi^{(x',y')}(x') - \psi^{(x',y')}(y') > w_n$ holds by \eqref{eq:Jn-non-empty}. Hence, we have
\[
c(x',y') - \bar\varphi(x') - \bar \psi(y') > w_n \quad \text{for all } (x',y') \notin I_{n-1},
\] 
which is a contradiction to $(\bar\varphi, \bar\psi) \in S_n$.

\emph{Step 4: Identification of the limiting point in \cite{CominettiSanMartin}.}\\
Note that Step 3 implies that the sequence $I_0 \subseteq I_1 \subseteq \ldots \subseteq I_n$ is strictly increasing, as long as $I_{n-1}\subsetneq \Xcal\times \Ycal$.
Because of strict monotonicity of the sets $I_n\subseteq \Xcal\times\Ycal$, for some $M\leq n_\Xcal\times n_\Ycal$ we have that $(\Xcal\cup\Ycal,I_M)$ is a connected graph. Let $M'$ be the first step where this happens. Then, by the same argument as in the proof of Proposition~\ref{prop:unique}-(ii), we can conclude that $S_{M'}=\{(\hat\varphi,\hat\psi)\}$.

Now, set $w_0=0$ and $J_0=I_0$.
In \cite{CominettiSanMartin} it is shown that the polytopes $S_0 \supseteq S_1 \supseteq \ldots \supseteq S_{M'}$ satisfy
\begin{equation}
\label{eq:characSN}
S_n =
\left\{ (\varphi, \psi): \begin{array}{l l} c(x,y) - w_j = \varphi(x) + \psi(y), & \text{ for all } j \in \{0, \ldots, n\},(x,y) \in J_j \\ 
c(x,y) - w_n \ge \varphi(x) + \psi(y), & \text{ for all } (x,y) \notin I_n \end{array} \right\}
\end{equation}
for all $n = \{0,1, \ldots, M'\}$, and that the pair $(\hat\varphi,\hat\psi)$ in $S_{M'}$ is indeed the limit of any sequence $(\varphi^{\eps_n}, \psi^{\eps_n})$ with $\eps_n \rightarrow 0$.

We are therefore left to show that $(\varphi^\ast, \psi^\ast)=(\hat\varphi,\hat\psi)$.

\emph{Step 5: Conclusion in the case of $G$ connected.}\\
If $G$ is connected (i.e. $N=1$), by Proposition~\ref{prop:unique}-(ii) we have that already $S_0=\{(\hat\varphi,\hat\psi)\}$ is a singleton, so the claim immediately follows. Hence, in the following we assume that $N>1$.

\emph{Step 6: Setting for the rest of the proof.}\\
We define $\varphi^\alpha:\Xcal\to\RR$ and $\psi^\alpha:\Ycal\to\RR$ by $\varphi^\alpha(x) = \varphi_n(x) + \alpha_n$ and $\psi^\alpha (y) = \psi_n(y) - \alpha_n$ for all $x \in X_n$ and $y \in Y_n$, $n\in\{1,\ldots,N\}$, with $\alpha_1=0$ and constants $\alpha_2,\ldots,\alpha_N\in\RR$ as in Theorem~\ref{thm:SetOfAllDuals}.
We write $\delta_1 < \ldots < \delta_L$ for the values of $\delta$ in line 6 of Algorithm~\ref{Algorithm:ConstructionT} that occur while iterating the while-loop. Moreover, we denote by $T_l$ the graph obtained in the while-loop in line~5 of  Algorithm~\ref{Algorithm:ConstructionT} for $\delta_l$, $l=1,...,L$. Finally, we set $\delta_0=0$, $\delta_{L+1}=+\infty$ and write $T_0$ for the graph on $\{1, \ldots, N\}$ with no edges.

\emph{Step 7: For any $j\in\{1,\ldots,M'\}$ and $l\in\{0,1,\ldots,L\}$,
if $w_j \in [\delta_l, \delta_{l+1})$, then $S_j$ is the set all pairs $(\varphi^\alpha, \psi^\alpha)$ with $\alpha$ such that $\alpha_1=0$ and:
\begin{itemize}
\item[(i)] for all edges $\{n,m\} \in T_l$,
\eqref{eq:DualLimitcEqual} holds; 
\item[(ii)] for all pairs $(n,m)\in \{1, \ldots, N\}^2$ such that there is no path in $T_l$ joining $n$ and $m$, then
\begin{align*}\label{eq.int.a}
\alpha_n - \alpha_m \in \left[- \min_{x \in X_m, y \in Y_n} c(x,y) - \varphi_m(x)-\psi_n(y) + w_j,  \min_{x \in X_n, y \in Y_m} c(x,y) - \varphi_n(x)- \psi_m(y) - w_j \right].
\end{align*} 
\end{itemize}}
We first prove by induction on $j$ that any pair $(\varphi^\alpha, \psi^\alpha)$ with $\alpha$ as in the claim satisfies the constraints in \eqref{eq:characSN} for $n=j$. Afterwards, we will show that any pair $(\varphi^\alpha,\psi^\alpha)$ where $\alpha$ does not satisfy the conditions in the claim violates at least one constraint in \eqref{eq:characSN}.

To prove the first statement, we note that for $j=0$ the claim immediately follows from Theorem~\ref{thm:SetOfAllDuals}. Now let $j \in \{1, \ldots, M'\}$ and  assume that the claim has been proved for $j-1$.
Let $\alpha$ be as in the claim. Let us first consider $x \in X_n$ and $y \in Y_n$, for some $n \in \{1, \ldots, N\}$. Then
\[
c(x,y) - \varphi^\alpha(i) - \psi^\alpha(y) = c(x,y) - \varphi_n(x) - \alpha_n - \psi_n(y) + \alpha_n = c(x,y) - \varphi_n(x) - \psi_n(y).
\] 
Hence, the difference $c(x,y) - \varphi^\alpha(x) - \psi^\alpha(y)$ equals some value $L(x,y)$ independent of $\alpha$. If $L(x,y) < w_j$, then, by the induction hypothesis,
there is $k<j$ such that $w_k = L(x,y)$ and $(x,y) \in J_k$. Hence,
 the equality in the first constraint in \eqref{eq:characSN} is satisfied.
If $L(x,y) = w_j$, then  $(x,y) \in J_j$ and the equality in the first constraint  in \eqref{eq:characSN} is satisfied.
Finally, if $L(x,y) > w_j$, then we have
\[
c(x,y) - w_j > c(x,y) - L(x,y)  = \varphi^\alpha(x) + \psi^\alpha(y),
\] 
hence $(x,y) \notin I_n$ and the second inequality
in \eqref{eq:characSN} is satisfied.

Now consider $x \in X_n$ and $y \in Y_m$ for $n,m\in \{1, \ldots, N\}$ with $n \neq m$ and such that they are connected by a path $n=n_1 n_2 \ldots n_k=m$ in $T_l$. Then
\begin{align*}
c(x,y) - \varphi^\alpha(x) - \psi^\alpha(y)
&= c(x,y) - \varphi_n(x) - \alpha_n - \psi_m(y) + \alpha_m \\
&= c(x,y) - \varphi_n(x) - \psi_m(y) -  \alpha_{n_1} + \alpha_{n_2} - \alpha_{n_2}+ \alpha_{n_3} - \ldots -\alpha_{n_{k-1}} + \alpha_{n_k}.
\end{align*} 
Since $\alpha$ satisfies the claim,  $\alpha_{n_i} - \alpha_{n_{i-1}}$ is constant for all $i \in \{2, \ldots, k\}$. Therefore,
the difference $c(x,y) - \varphi^\alpha(x) - \psi^\alpha(y)$ equals some value $L(x,y)$. If $L(x,y)<w_j$, then $L(x,y)\le w_{j-1}$ and hence, by the induction hypothesis, the first constraint in \eqref{eq:characSN} is satisfied. If $L(x,y)=w_j$, we obtain $(x,y) \in J_j$ and again the first constraint in \eqref{eq:characSN} is satisfied. If $L(x,y)>w_j$, then we have
\[
c(x,y) - w_j > c(x,y)) - L(x,y) = \phi^\alpha (x) + \psi^\alpha(y),
\] 
thus $(x,y) \notin I_n$ and the second constraint in \eqref{eq:characSN} is satisfied.

Finally, we consider $x \in X_n$ and $y \in Y_m$ such that $n$ and $m$ are not joined by a path in $T_l$. In this case, by construction of $T_l$ we have that $\delta_{n,m} \ge \delta_{l+1} > \delta_l$. In particular, $\delta_{n,m} > w_j$. Hence, the difference $c(x,y) - \varphi^\alpha(x) - \psi^\alpha(y)$ can take different values depending on $\alpha$, thus $(x,y) \notin I_n$. 
By choice of $\alpha$ we have
\begin{align*}
c(x,y) - \varphi^\alpha(x) - \psi^\alpha(y)&= c(x,y) - \varphi_n(x) - \alpha_n - \psi_m(y) + \alpha_m \\
&= c(x,y) - \varphi_n(x) - \psi_m(y) - (\alpha_n-\alpha_m) \\
&\ge c(x,y)- \varphi_n(x) - \psi_m(y) - \min_{x' \in X_n, y' \in Y_m} \left(c(x',y') - \varphi_n(x') - \psi_m(y') \right) + w_j \\
&\ge w_j,
\end{align*} 
hence the second constraint in \eqref{eq:characSN} is satisfied.
This concludes the proof of the fact that, for all $\alpha$ satisfying the claim, the pair $(\varphi^\alpha,\psi^\alpha)$ belongs to $S_j$.

In order to conclude the proof of the claim, it remains to prove that for any $\alpha$ not satisfying the constraints of the claim we have $(\varphi^\alpha, \psi^\alpha) \notin S_j$.
Assume first that $(\varphi^\alpha, \psi^\alpha)$ is such that there are $n,m \in \{1, \ldots, N\}$ such that $(n,m) \in T_l$ and $
\alpha_n - \alpha_m > L_{n,m}$, with $L_{n,m}$ given in \eqref{eq:DualLimitcEqual}.
Note that by construction of $T_l$ we have  $\delta_{n,m} \le \delta_l \le w_l$. Let moreover $x' \in X_n$ and $y' \in Y_m$ be such that
\[
c(x',y') - \varphi_n(x') - \psi_m(y') = \min_{x \in X_n, y \in Y_m} c(x,y) -\varphi_n(x)- \psi_m(y).
\] Then 
\begin{align*}
c(x',y') - \varphi^\alpha(x') - \psi^\alpha(y') &= c(x',y') - \varphi_n(x') - \psi_m(y') - \alpha_n + \alpha_m \\
&< \min_{x \in X_n, y \in Y_m} c(x,y) - \varphi_n(x)- \psi_m(y) -  \frac{1}{2}  \min_{x \in X_n, y \in Y_m} \big( c(x,y) - \varphi_n(x) - \psi_m(y) \big) \\
&\quad + \frac{1}{2} \min_{x \in X_m, y \in Y_n} \big( c(x,y) - \varphi_m(x) - \psi_n(y) \big) \\
&= \delta_{n,m} \le  w_l,
\end{align*} 
thus the constraints in \eqref{eq:characSN} are not satisfied for $(x',y')$. One can proceed in an analogous way for $(\varphi^\alpha, \psi^\alpha)$ such that there are $n,m \in \{1, \ldots, N\}$ with $(n,m) \in T_l$ and $\alpha_n - \alpha_m < L_{n,m}$.

Now assume that $(\varphi^\alpha, \psi^\alpha)$ is such that there are $n,m \in \{1, \ldots, N\}$ with $n \neq m$ such that there is no path from $n$ to $m$ in $T_l$ and that 
\[
\alpha_n - \alpha_m > \min_{x \in X_n, y \in Y_m} c(x,y) - \varphi_n(x)- \psi_m(y) - w_j.
\]
Let $x' \in X_n$ and $y' \in Y_m$ be such that 
\[
c(x',y') - \varphi_n(x') - \psi_m(y') = \min_{x \in X_n, y \in Y_m} c(x,y) -\varphi_n(x)- \psi_m(y).
\] Then 
\begin{align*}
c(x',y') - \varphi^\alpha(x') - \psi^\alpha(y') &= c(x',y') - \varphi_n(x') - \psi_m(y') - \alpha_n + \alpha_m \\
&< \min_{x \in X_n, y \in Y_m} c(x,y)- \varphi_n(x)- \psi_m(y) - \min_{x \in X_n, y \in Y_m} c(x,y)- \varphi_n(x)- \psi_m(y) + w_j  \\
&= w_j,
\end{align*} hence the constraints in \eqref{eq:characSN} are not satisfied. Finally, we consider the case of
$(\varphi^\alpha, \psi^\alpha)$ such that there are $n,m \in \{1, \ldots, N\}$, $n \neq m$, with no path from $n$ to $m$ in $T_l$ and 
\[
\alpha_n -\alpha_m < - \min_{x \in X_m, y \in Y_n} \big(c(x,y)- \varphi_m(x)-\psi_n(y)\big) + w_j.
\]
Note that this inequality is equivalent to 
\[
\alpha_m -\alpha_n > \min_{x \in X_m, y \in Y_n} \big(c(x,y)- \varphi_m(x)-\psi_n(y)\big) - w_j,
\] 
thus we arrive again at a violation of the constraints in \eqref{eq:characSN}.

\emph{Step 8. Conclusion.} \\
By Step 7, we note that for any $j \in \{1, \ldots, M'\}$ such that $w_j \ge \delta_L$ we have $S_j = \{(\varphi^\ast, \psi^\ast)\}$. We will now show that for any $j \in \{1, \ldots, M'\}$ such that $w_j \in  [\delta_{L-1}, \delta_L)$ we have $|S_j|>1$. From this, since the sets $(S_j)_{j \in \{1, \ldots, M'\}}$ are decreasing, and $S_{M'}=\{(\hat\varphi,\hat\psi)\}$, it will follow that $|S_j|>1$ for all $j \in \{1, \ldots, M'\}$ with $w_j < \delta_L$, while $w_{M'} = \delta_L$ with $S_{M'} = \{(\varphi^\ast, \psi^\ast)\}$, which then concludes the proof.

Hence, let us now consider $j \in \{1, \ldots, M'\}$ such that $w_j \in  [\delta_{L-1}, \delta_L)$ and let us prove that the set $S_j$ consists of at least two elements. First, note that the set $S_j$ contains the point $(\varphi^\ast, \psi^\ast)$, again since the sets $(S_j)_{j \in \{1, \ldots, M'\}}$ are decreasing. This means that the unique constants $(\alpha_1, \alpha_2, \ldots, \alpha_N)$ from Step 1 satisfy (i)-(ii) of Step 7  for $w_j\in  [\delta_{L-1}, \delta_L)$ and $T_{L-1}$.

Note also that the graph $T_{L-1}$ is not connected. Indeed, since the algorithm successively adds edges, we know that, as the algorithm has not stopped ($L-1<L$), the graph $T_{L-1}$ has less edges than $T_L$. Now, by \cite[Theorem 1.5.1]{DiestelGT}, any tree is minimally connected, which means that, whenever at least one edge is removed, the resulting subgraph is not connected. Hence, $T_{L-1}$ is not connected.
Let us enumerate the vertices of $T_{L-1}$ as $v_1, \ldots, v_N$ such that there is $K \in \{1, \ldots, N-1\}$ for which $\{v_{K+1}, \ldots, v_N\}$ is a connected component of $T_{L-1}$. Note that the set $\{v_1, \ldots, v_K\}$ can contain one or more components. Since $T_{L-1}$ is a subgraph of a tree and $\{v_{K+1}, \ldots, v_N\}$ is connected, again by \cite[Theorem 1.5.1]{DiestelGT}, there is a unique path from $v_{n}$ to $v_m$  for all $n,m \ge K+1$. Let $v_n=v^{(1)}v^{(2)}\ldots v^{(l)}=v_m$ be this path. Then define
\[
    \hat{L}_{v_n,v_m} =  L_{v^{(1)} v^{(2)}} + \ldots + L_{v^{(l-1)}v^{(l)}}.
\] 
Any $\tilde{\alpha}=(\tilde{\alpha}_1,\ldots,\tilde{\alpha}_N)$ with 
\[\tilde{\alpha}_v = \alpha_v \text{ for all }v \in \{v_1, \ldots, v_K\},   \quad \tilde{\alpha}_{v_{K+1}} = \beta, \quad \text{ and }\, \tilde{\alpha}_v = \beta + \hat{L}_{v_n,v_{K+1}} \text{ for all } v \in \{v_{K+2}, \ldots,  v_N\},
\] 
for some 
\begin{align}
    \label{eq:DualLimitConcl}
    \begin{split}
        \beta \in &\left[ \max_{m \le K, n \ge K+1} - 
        \left( \min_{x \in X_{v_m}, y \in Y_{v_n}} c(x,y) - \varphi_{v_m}(x) - \psi_{v_n}(y) \right) + w_j - \hat{L}_{v_n,v_{K+1}} + \alpha_{v_m}, \right. \\
        &\quad \left. \min_{m \le K, n \ge K+1} \left( \min_{x \in X_{v_n}, y \in Y_{v_m}} c(x,y) - \varphi_{v_n}(x) - \psi_{v_m}(y)\right) -w_j - \hat{L}_{v_n, v_{K+1}} + \alpha_{v_m} \right],
    \end{split}
\end{align} 
satisfies the constraints in Step 7.  Indeed, \eqref{eq:DualLimitConcl} is equivalent to the constraint (ii) in Step 7 for $v_n, v_m$ such that $m \le K$ and $n \ge K+1$. That the remaining constraints are satisfied (i.e. (i) for all $\{n,m \} \in T_{L-1}$ and (ii) for $v_n, v_m$ such that $n,m \le K$ and $v_n, v_m$ such that $n,m \ge K+1$) follows from the fact that $\alpha$ satisfies (ii) in Step 7 and since we have $\tilde{\alpha}_{v_n} - \tilde{\alpha}_{v_m} = \alpha_{v_n} - \alpha_{v_m}$ in all described cases.

We will now show that the interval in \eqref{eq:DualLimitConcl} has non-empty interior, which implies that $|S_j|>1$. 
For this we note that, since $S_j=\{(\phi^\ast, \psi^\ast)\}$ for $w_j = \delta_L$, then $\beta = \alpha_{v_{K+1}}$ satisfies the constraint \eqref{eq:DualLimitConcl} with $\delta_L$ instead of $w_j$. Hence,
\begin{align*}
&\max_{m \le K, n \ge K+1} - \left( \min_{x \in X_{v_m}, y \in Y_{v_n}} c(x,y) - \varphi_{v_m}(x) - \psi_{v_n}(y) \right) + w_j - \hat{L}_{v_n,v_{K+1}} + \alpha_{v_m}, \\
&<\max_{m \le K, n \ge K+1} - \left( \min_{x \in X_{v_m}, y \in Y_{v_n}} c(x,y) - \varphi_{v_m}(x) - \psi_{v_n}(y) \right) + \delta_L - \hat{L}_{v_n,v_{K+1}} + \alpha_{v_m}, \\
&\le  \min_{m \le K, n \ge K+1} \left( \min_{x \in X_{v_n}, y \in Y_{v_m}} c(x,y) - \varphi_{v_n}(x) - \psi_{v_m}(y)\right) -\delta_L - \hat{L}_{v_n, v_{K+1}} + \alpha_{v_m} \\
&<  \min_{m \le K, n \ge K+1} \left( \min_{x \in X_{v_n}, y \in Y_{v_m}} c(x,y) - \varphi_{v_n}(x) - \psi_{v_m}(y)\right) -w_j - \hat{L}_{v_n, v_{K+1}} + \alpha_{v_m}
\end{align*}
which shows that the interval in \eqref{eq:DualLimitConcl} has a non-empty interior. Therefore, $|S_j|>1$ and the claim follows.
\end{proof}

\begin{remark}
The construction of the centroid in Cominetti and San Mart\'{i}n~\cite{CominettiSanMartin} is informally described as tightening all non-saturated constraints until some of them become binding. 
This idea is also the basis of our construction. 
The set of all dual optimizers is described by the set of all constants $\alpha_1, \ldots, \alpha_N$ satisfying the constraints in Theorem~\ref{thm:SetOfAllDuals}-(iii).  These constraints  require that $\alpha_n-\alpha_m$ lies in a particular interval. Indeed, the upper and lower bound for the difference $\alpha_n - \alpha_m$ is given by the value where for some pair in $X_m \times Y_n$ and in $X_n \times Y_m$, respectively, the constraint becomes binding. The term $2\delta_{n,m}$ now describes the difference of the upper and the lower bound, i.e. the width of the admissible interval. Hence, if $w_j=\delta_{n,m}$ the constraint will become binding.

In our algorithm, we successively add edges to a graph on $\{1, \ldots, N\}$. These edges  (non-redudantly) describe that the difference $\alpha_n - \alpha_m$ is fixed to a certain value, namely, the one given by \eqref{eq:DualLimitcEqual}.  The edges that we can add come from a candidate set, which is shrinking. At first, it is the set of all pairs $(n,m) \in [\{1, \ldots, N\}]^2$. Once we add an edge $(n,m)$ to the graph $T$, we delete all pairs $(n',m')$ where the value of $\alpha_{n'}-\alpha_{m'}$ is already fixed to a certain value, i.e. we delete all edges for which the pair $(n',m')$ does not impose an additional constraint on the solution set. These pairs are exactly those that are connected in $T_l$, since in this case 
$\alpha_{n'} - \alpha_{m'} = \alpha_{n'} - \alpha_{n_1} + \alpha_{n_1} - \alpha_{n_2} + \ldots - \alpha_{m'}$ is already fixed to a certain value, as $\{n^i,n^{i+1}\} \in T_l$ for all $i \in \{[0, \ldots, k-1\}$.
\hfill $\diamond$
\end{remark}

\begin{example}
\label{ex3}
We consider a slightly modified version of Example~\ref{ex2}. Namely, we set
\[
\hat{c} (x,y) = c(x,y) +2 \cdot 1_{\{(3,4)\}}(x,y) + 1_{\{(4,3)\}}(x,y).
\] 
The primal optimizer $\gamma$ from Example~\ref{ex2} is again an optimizer for the problem $\OT(\mu, \nu, \hat{c})$, but now we have $G=G^\gamma$. Nonetheless, the functions $\varphi_n$ and $\psi_n$, $n \in \{1,2,3\}$, described in Example~\ref{ex2} are still the unique dual optimizers for the subproblems on the connected components. Any dual optimizer can be represented as in Theorem~\ref{thm:SetOfAllDuals} where the constant $\alpha=(\alpha_1, \alpha_2, \alpha_3)$ now satisfies 
\[
0 \le \alpha_1 - \alpha_2 \le 2, \, 0 \le \alpha_1 - \alpha_3 \le 1 \text{ and } -2 \le \alpha_2 - \alpha_3 \le 1.
\] 

We now derive the limit $(\varphi^\ast, \psi^\ast)$ of the dual optimizers of the entropic optimal transport problems using Algorithm~\ref{Algorithm:ConstructionT} and Theorem~\ref{thm:LimitEntropic}. 
We first note that 
\[
\delta_{1,2} = 0.5 \cdot (2 + 0) = 1, \; \delta_{1,3} = 0.5 \cdot (1 + 0) = 0.5 \text{ and } \delta_{2,3} =0.5 \cdot (1 + 2) = 1.5.
\] 
Hence, the algorithm selects $\{1,3\}$ as the first edge of $T$ and $\{1,2\}$ as the second edge. Then the constant $\alpha$ satisfies
\[
\alpha_1 - \alpha_2 = 0.5 \cdot 2 - 0.5 \cdot 0 =1 \text{ and } \alpha_1 - \alpha_3 = 0.5 \cdot 1 - 0.5 \cdot 0 = 0.5,
\] which, using the convention $\alpha_1=0$, yields $\alpha_2 = -1$ and $\alpha_3=-0.5$. Therefore, the limit $(\varphi^\ast, \psi^\ast)$ of the dual entropic optimizers  reads as
\begin{align*}
    &\varphi^\ast(1) = 0, \; \varphi^\ast(2) = -1, \; \varphi^\ast(3) = -1, \; \varphi^\ast(4) = -0.5 \\
    &\psi^\ast(1)=1,  \; \psi^\ast(2)=2, \; \psi^\ast (3)=1,\; \psi^\ast(4) = 1.5, \; \psi^\ast(5) = 1.5. 
\end{align*}\hfill $\diamond$
\end{example}

\section{Stackelberg-Cournot-Nash Equilibria.}
\label{sec:SCNE}
In this section we consider a game between a principal and a population of agents, that is a
Stackelberg version of the problem considered by
Blanchet and Carlier \cite{BlanchetCN}. We provide existence results together with a characterization of equilibria. We then study approximation by regularization and conclude with a numerical example.

\subsection{Problem Formulation.}
We consider a continuum of agents (population), characterized by a finite number of types,  that need to choose among a finite number of actions, and a principal who determines the additional costs an agent faces for choosing a certain action. 
To set this problem in the optimal transport setting of Section~\ref{sect.fOT}, we let $\Xcal=\{x_1,\ldots,x_{n_\Xcal}\}$ be the set of all types and $\Ycal=\{y_1,\ldots,y_{n_\Ycal}\}$ the set of all actions. A distribution of types, $\mu\in\Pcal(\Xcal)$, is fixed, and the optimal distribution of actions will be found in equilibrium among all $\nu\in\Pcal(\Ycal)$. With an abuse of notation, we identify the distributions $\mu$ and $\nu$ with their respective probability vectors in the $n_\Xcal$- and $n_\Ycal$-simplex, $\Delta_{n_{\Xcal}}$ and $\Delta_{n_{\Ycal}}$ resp., i.e., $\mu=(\mu_1,\ldots,\mu_{n_\Xcal})\equiv (\mu(x_1),\ldots,\mu(x_{n_\Xcal}))\in\Delta_{n_{\Xcal}}=\{\omega\in[0,1]^{n_\Xcal}:\sum_{n=1}^{n_\Xcal}\omega_n=1\}$, and analogously for $\nu$. 
For each $\nu\in\Delta_{n_{\Ycal}}$, a coupling $\gamma\in\Pi(\mu,\nu)\subseteq\Pcal(\Xcal\times\Ycal)$ describes the strategy of the agents, with the interpretation that $\gamma_{ij}/\mu_i$ is the probability that an agent of type $x_i$ chooses action $y_j$.
Again with an abuse of notation, we use $\gamma$ also to denote the matrix in the simplex $\Delta_{n_\Xcal\times n_\Ycal}$ with entries $\gamma_{ij}=\gamma(x_i,y_j)$, for $i\in\{1,\ldots,n_\Xcal\}, j\in\{1,\ldots,n_\Ycal\}$.

The cost of each agent does not only depend on its own type and action, but also on the actions of all other agents in a mean-field sense, i.e. it will not depend on single choices of other agents, but on the distribution $\nu$ of their actions (the continuum of agents allows us to consider as indistinguishable the distribution of actions of all agents and that of all agents except one). 
Specifically, for a distribution of actions $\nu\in\Delta_{n_{\Ycal}}$ and a vector of costs $k= (k_j)_{j\in\{1,\ldots,n_\Ycal\}}\in K$ chosen by the principal from a fixed subset $K \subseteq \mathbb{R}^{n_\Ycal}$, the cost of an agent of type $x_i$, $i\in\{1,\ldots,n_\Xcal\}$, choosing action $y_j$, $j\in\{1,\ldots,n_\Ycal\}$, is given by
\begin{equation}\label{eq.totcost}
C_{ij}[\nu,k] := c_{ij} + k_j + f_j(\nu_j) + \sum_{a=1}^{n_\Ycal} \theta_{aj}\nu_a.
\end{equation}
Here: $c=(c_{ij})_{i\in\{1,\ldots,n_\Xcal\}, j\in\{1,\ldots,n_\Ycal\}}\in\RR^{n_\Xcal\times n_\Ycal}$ takes care of the part of the cost depending on both type and action of such agent; $k_j$ is paid for action $j$ independently of the type and of other agents actions; and last two components consider the interactions with the other agents, with $f_j: [0,1] \rightarrow \mathbb{R}$ nondecreasing and continuous functions, reflecting the fact that choosing a more popular action (within agents of same type) is more costly, and $(\theta_{aj})_{a,j\in\{1,\ldots,n_\Ycal\}}\in\RR^{n_\Ycal\times n_\Ycal}$ symmetric matrix so that the last term reflects the cost related to actions of agents of different type.
Finally, the principal faces a cost $G(\nu,k)$ that depends on the chosen vector of costs $k$ and on the agents' actions through the distribution $\nu$. The function $G: \Delta_{n_\Xcal\times n_\Ycal}\rightarrow \mathbb{R}$ is assumed to be continuous.
A possible interpretation of the different role played by the terms in \eqref{eq.totcost} is illustrated by the following classical example.
\begin{example} \label{ex:SCNE}
Consider a big company with many employees for which vacation times have to be coordinated. The possible vacation times are $\Ycal = \{1, \ldots, n_{\Ycal}\}$, which we could interpret as weeks. The employees differ through their preferences for these time slots because they have kids in school, prefer travelling in summer or winter, etc. We assume that there  are finitely many types $\Xcal = \{1, \ldots, n_\Xcal\}$. The cost of the individual agents is given by 
\[
C_{ij}[\nu,k] := c_{ij} + k_j + f_j(\nu_j) + \sum_{a=1}^{n_\Ycal} g(|a-j|) \nu_a,
\] 
where $g$ is a decreasing function and $f_j,  j=1, \ldots, n_\Ycal$ are strictly increasing and continuous functions. The components are interpreted as follows: the first is the cost of taking vacation in week $j$ for an agent of type $i$; $k_j$ is an additional cost charged by the employer for agents that pick week $j$; the last two terms capture the effect that the workload is increasing the more agents are on holiday in the same week (captured by $f_j(\nu_j)$) or in the weeks that are close (captured by $g(|a-j|) \nu_a$, since $g(|a-j|)$ is smaller the larger the distance of $a$ and $j$ is).
The principal's cost function reads as $G(\nu) = \sum_{j=1}^n \nu_j^2$, expressing the preference that not too many employees are on vacation at the same time. This cost function does not depend on $k$ reflecting the fact that employers can increase or reduce the costs of the agents, which are measured in utility, by measures which are not costly for themselves. \hfill$\diamond$
\end{example}

In what follows we will use $p_1$ and $p_2$ for the projections into first and second marginal of a measure, so that $\gamma\in\Pi(\mu,\nu)$ satisfies $p_1\#\gamma=\mu$ and $p_2\#\gamma=\nu$. We will denote by $\Pi(\mu,\cdot)$ the set of measures $\gamma$ with $p_1\#\gamma=\mu$ and any second marginal, and by $\Pi(\cdot,\nu)$ the set of measures $\gamma$ with $p_2\#\gamma=\nu$ and any first marginal. Similar notation will be used when one of the marginal is not fixed in the $\OT$ problem.

\begin{definition}[SCNE]\label{def.SCNE}
For $k \in K$, a strategy $\gamma^k$ is said \emph{optimal for $k$}, or a \emph{Cournot-Nash equilibrium} (CNE) w.r.t. $k$, if it is optimal for the problem
\begin{equation}\label{eq.CNE}
\inf_\gamma\, C[\nu^k,k] \cdot \gamma = \inf_\gamma \sum_{i=1}^{n_\Xcal}\sum_{j=1}^{n_\Ycal}
C_{ij}[\nu^k,k]\gamma_{ij},
\end{equation}
where $\nu^k=p_2\#\gamma^k$, and the minimization is run over all $\gamma\in \Pi(\mu,\cdot)$.

A pair $(\gamma^{k^*}, k^*)\in\Pi(\mu,.)\times K$ is a \emph{Stackelberg-Cournot-Nash equilibrium (SCNE)} if it satisfies the two following conditions:
\begin{itemize}
\item[(i)] $\gamma^{k^*}$ is a CNE w.r.t. $k^*$,
\item[(ii)] $G(\nu^{k^*},k^*) \le G(\nu^k, k)$ for all $k\in K$ and $\gamma^k$  CNE w.r.t. $k$. 
\end{itemize}
\end{definition}
The fact that, for a fixed $k\in K$, the solution to problem \eqref{eq.CNE} is a Nash equilibrium for the agents, is easily seen considering that we are in a setting with a continuum of agents, so that a change of action by one of them would not change the distribution of actions; see  
\cite{acciaio2021cournot}. This gives condition (i) of the SCNE. Condition (ii) expresses the Stackelberg equilibrium in a principal-agent game, that is, the situation where the principal optimizes over a set of possible choices (here $k\in K$), knowing how agents would optimally act w.r.t. each choice (here $\gamma^k$).

\subsection{The Optimization Problem of the Agents.} \label{sec:game_continuum}
In what follows we will relate the optimization problem for the agents to an equivalent variational problem related to optimal transport.
Let us define the energy function $\mathcal{E}:\Delta_{n_{\Ycal}}\to\RR$ as
\[
\mathcal{E}[\nu] := \sum_{j=1}^{n_\Ycal} F_j(\nu_j) + \frac{1}{2} \sum_{a,j=1}^{n_\Ycal} \theta_{aj}\nu_a\nu_j,
\] 
with $F_j(t):= \int_0^t f_j(s) \de s$.
Then the variational problem of interest is given by
\begin{equation}
\label{eq:VariationalProblem}
\inf_{\nu \in \Pcal(\Ycal)} \left\{ \OT(\mu,\nu,c) + k \cdot \nu + \mathcal{E}[\nu]\right\}.
\end{equation}

As for standard Cournot-Nash games, we can relate Cournot-Nash equilibria to the variational problem \eqref{eq:VariationalProblem}.

\begin{prop}[\cite{acciaio2021cournot},Theorem 3.4]
\label{thm:PopulationGame}
Assume that $\mathcal{E}$ is convex. Then $\gamma^k\in\Pi(\mu,\cdot)$ is a CNE w.r.t. $k$ if and only if $\nu^k=p_2\#\gamma^k$ solves \eqref{eq:VariationalProblem} and $\gamma^k$ is an optimizer for $\OT(\mu, \nu^k,c)$.
\end{prop}

As in \cite{BlanchetCNFinite}, we define 
\begin{align*}
\oOT(\mu, \nu, c) := \begin{cases}
\OT(\mu,\nu,c), &\text{if } \nu \in \Pcal(\Ycal) \\
\infty, &\text{else.}
\end{cases}
\end{align*} 
The reason for this is to have a function $\oOT(\mu, \cdot, c)$ defined on the whole $\RR^{n_\Ycal}$, so that one can apply classical results from convex analysis.

A special role in our results will be played by the subdifferential of $F=\OT(\mu, \cdot , c)$ or $F=\oOT(\mu, \cdot, c)$, which is defined by
\[
\partial F(\nu)=\left\{
f\in \RR^{n_\Ycal} : F(\nu) - f\cdot \nu \leq F(\eta) - f\cdot \eta,\quad \forall\ \eta\in\Pcal(\Ycal)
\right\}.
\]
Note that, for $\nu\notin\Pcal(\Ycal)$ we have $\partial_\nu \oOT(\mu,\nu, c)=\emptyset$, while for $\nu\in\Pcal(\Ycal)$ we have $\partial_\nu \oOT(\mu,\nu, c)=\partial_\nu \OT(\mu,\nu, c)$. 
As an immediate consequence, we obtain the following result.
\begin{prop}
\label{prop:NecessarySufficient}
Assume that $\mathcal{E}$ is convex. Then $\nu\in\Pcal(\Ycal)$ is a minimizer of \eqref{eq:VariationalProblem} if and only if 
\[
0 \in \partial_\nu \OT(\mu,\nu, c) + k + \nabla \mathcal{E}[\nu].
\] 
If $\mathcal{E}$ is strictly convex, then there is a unique optimizer of \eqref{eq:VariationalProblem}.
\end{prop}

\subsection{The Optimization Problem of the Principal.}
For any fixed vector of costs $k\in K$, the optimization problem of the population, that is \eqref{eq.CNE}, has been reduced to the variational problem \eqref{eq:VariationalProblem}. 
Let us write 
\[
\text{BR}: K \rightarrow 2^{\Pcal(\Ycal)}
\]
for the set-valued map that maps $k$ to the set of all optimizers of \eqref{eq:VariationalProblem}.
By Proposition~\ref{prop:NecessarySufficient}, whenever $\mathcal{E}$ is strictly convex, the optimizer $\nu^k$ of \eqref{eq:VariationalProblem} is unique. Hence, in this case the map $\text{BR}$ is a function.

\begin{theorem}
\label{thm:BRContinuous}
Assume that $\mathcal{E}$ is convex.
Then the map $\text{BR}$ has a closed graph.
\end{theorem}

\begin{proof}
In order to show that $\{(k, \nu): \nu \in \text{BR}(k)\}$ is closed, it suffices to prove that, for any sequence $(k^n)_n$ that converges to $k$ and any sequence $(\nu^n)_n$ that converges to $\nu$, with $\nu^n \in \text{BR}(k^n)$, we have that $\nu \in \text{BR}(k)$. By Proposition~\ref{prop:NecessarySufficient}, the values $m^n := -k^n - \nabla \mathcal{E}[\nu^n]$ satisfy $m^n \in \partial_\nu \oOT(\mu,\nu^n, c)$. Since $(k^n)_n$ and $(\nu^n)_n$ are converging sequences and $\mathcal{E} \in \mathcal{C}^1$, we obtain that $m^n$ converges towards $m:=-k-\nabla \mathcal{E}[\nu]$. Since the map $\OT(\mu, \cdot, c)$ is lower semicontinuous (see \cite[p.13]{acciaio2021cournot}), the subdifferential $\partial_\nu \oOT(\mu, \cdot, c)$ is upper semicontinuous (see \cite[p. 55]{VillaniTopics}). Thus, $m \in \partial_\nu \oOT(\mu, \nu, c)$. This is equivalent to
\[
0 \in \partial_\nu \oOT(\mu, \nu, c) + k + \nabla \mathcal{E}[\nu],
\] 
which shows that $\nu$ is indeed the optimal response of the population to the vector of costs $k$, i.e. $\nu \in \text{BR}(k)$.
\end{proof}

\begin{prop}
Let $\mathcal{E}$ be convex and $K$ be compact. Then a SCNE exists.
\end{prop}

\begin{proof}
The statement follows directly from Theorem~\ref{thm:BRContinuous}, since the latter implies that the infimum
\[ 
\inf_{\{(k,\nu): k \in K, \nu \in \text{BR}(k)\}} G(\nu,k)
\] 
is attained.
\end{proof}

\begin{remark}
We highlight that this result crucially relies on the link of optimal transport and the Cournot-Nash equilibria. Indeed, standard techniques from game theory yield only an existence result and only for the game without the principal. However, for our proof the characterization of the equilibria as solutions to a variational problem is a cornerstone of the analysis. \hfill$\diamond$
\end{remark}

As a next step we provide a helpful reformulation of the optimization problem for the principal.

\begin{theorem}
\label{thm:OptimizationPrincipal}
Assume that $\mathcal{E}$ is convex. Then it holds that
\[
\inf_{k \in K} \inf_{\nu \in \text{BR}(k)} G(\nu,k) = \inf_{\nu \in \Pcal(\Ycal)} \inf_{k \in \left(-\partial_\nu\OT(\mu,\nu, c) - \nabla \mathcal{E}(\nu) \right) \cap K} G(\nu,k),
\] 
with the convention that $\inf \emptyset = \infty$.
\end{theorem}

\begin{proof}
By \cite[Theorem 23.5]{Rockafellar}, $k \in -\partial_\nu \OT(\mu, \nu, c) -\nabla \mathcal{E}[\nu]=-\partial_\nu \oOT(\mu, \nu, c) -\nabla \mathcal{E}[\nu]$ is equivalent to the fact that the function 
\[
\tilde{\nu} \mapsto k\cdot \tilde{\nu} + \oOT(\mu, \tilde{\nu}, c) + \mathcal{E}[\tilde{\nu}]
\] 
achieves its infimum for $\tilde{\nu}=\nu$.
\end{proof}

An important class of models are those where the reward of the principal reads $G(\nu,k)=G(\nu)$ for some function $G: \Delta_{n_{\Ycal}} \rightarrow \mathbb{R}$. 
This is reasonable if the costs of the agents are measured in terms of utility and the principal can influence these costs in such a way that it is costless for itself, see Example~\ref{ex:SCNE}.

For this class of models, we obtain an existence result that even yields the possibility to compute CNEs.

\begin{corollary}
\label{cor:SCNEex}
Assume that $\mathcal{E}$ is convex.
Let $G(\nu,k)=G(\nu)$ for all $\nu \in \Delta_{n_{\Ycal}}$, $k \in K$. 
Assume that $\nu^\ast$ is a minimizer of $G$ and that $\gamma^\ast$ is an optimizer for $\OT(\mu,\nu^\ast, c)$.
Then, for any $k^\ast \in \left(-\partial_\nu\OT(\mu,\nu^\ast,c) - \nabla \mathcal{E}(\nu^\ast) \right) \cap K$, the pair $(k^\ast, \gamma^\ast)$ is a SCNE.
In particular, if $\nu^\ast \in \text{ri}(\Delta_{n_{\Ycal}})$ and $K=\mathbb{R}^{n_\Ycal}$, then a SCNE exists.
\end{corollary}

\begin{proof}
For the first part of the claim, by Theorem~\ref{thm:OptimizationPrincipal} it suffices to show that $(k^\ast, \gamma^\ast)$ is an optimizer of
\[\inf_{\nu \in \Pcal(\Ycal)} \inf_{k \in \left(-\partial_\nu\OT(\mu,\nu, c) - \nabla \mathcal{E}(\nu) \right) \cap K} G(\nu).\] 
Since the function $G$ does not depend on $k$, this means (recall the convention $\inf \emptyset =\infty$) that the optimizer of $G(\nu)$ over all $\nu \in \Pcal(\Ycal)$ such that $\left(-\partial_\nu\OT(\mu,\nu, c) - \nabla \mathcal{E}(\nu) \right) \cap K$ is non-empty. By assumption, $\nu^\ast$ is such a minimizer and $k^\ast$ is one minimizer for the inner infimum.  The second part of the claim follows since $\oOT$ is a proper closed convex function and such functions are subdifferentiable in the relative interior of the domain $\Delta_{n_{\Ycal}}$; see \cite[Theorem 23.4]{Rockafellar}.
\end{proof}

\subsection{The Subdifferential of $\OT$.}
Given the reformulation of the principal's problem in Theorem~\ref{thm:OptimizationPrincipal}, understanding the subdifferential of the optimal transport problem turns out to be an essential step 
in order to solve the SCNE problem. We will see in Theorem~\ref{thm.subdiffOT} below that this is tightly related to the optimizers of the dual problem.

Recall that the so-called $c$-transform of the function $k$, denoted by $k^c$, is given by
\[
k^c_i:=\min_{j\leq n_\Ycal}c_{ij}-k_j,\quad i=1,\ldots,n_\Xcal.
\]
This clearly satisfies the constraint
\[
k^c_i+k_j\leq c_{ij}\quad \forall i\leq n_\Xcal, j\leq n_\Ycal,
\]
so that the pair $(k^c,k)$ is feasible for the dual problem $\DOT(\mu,\nu,c)$. 

\begin{theorem}\label{thm.subdiffOT}
Let $\nu \in \Delta_{n_{\Ycal}}$.
Then $k\in K$ satisfies $k \in \partial_\nu \OT(\mu,\nu, c)$ if and only if $(k^c,k)$ is an optimizer for $\DOT(\mu,\nu, c)$, i.e.
\[
k^c\cdot\mu+k\cdot\nu=\max\left\{\varphi\cdot\mu+\psi\cdot\nu: \varphi\in\RR^{n_\Xcal},\psi\in\RR^{n_\Ycal}, \varphi_i+\psi_j\leq c_{ij}\ \forall i\leq n_\Xcal, j\leq n_\Ycal\right\}.
\]
\end{theorem}
We remark that this result has been proved for general probability measures supported on a compact subset of $\mathbb{R}^d$ in \cite[Proposition 7.17]{santambrogio2015optimal}. Here we present a simpler proof for our discrete setting.

\begin{proof}
Let us first fix $k \in \partial_\nu \OT(\mu,\nu,c)$. Then we have for any $\eta \in \Delta_{n_{\Ycal}}$ and any $\gamma \in \Pi(\mu,\eta)$ that
\begin{equation}
\label{eq:SubgradientKantorovichProof}
\OT(\mu,\nu, c) - k \cdot \nu \le \OT(\mu, \eta, c) - k\cdot \eta \le c \cdot\gamma - k\cdot \eta. 
\end{equation} 
Now, for any $i\leq n_\Xcal$, choose $j^{(i)}$ such that 
\[
j^{(i)} \in \text{argmin}_{j\leq n_\Ycal} \{c_{ij}-k_j\}
\] 
and set $\gamma\in\Pi(\mu,\cdot)$ via
\[
\gamma_{ij} = \begin{cases}
0, & \text{for $j \neq j^{(i)}$} \\
\mu_i, & \text{for $j = j^{(i)}$}.
\end{cases}
\] 
Letting $\eta=p_2\#\gamma$, by \eqref{eq:SubgradientKantorovichProof} we get
\begin{align*}
\OT(\mu,\nu, c) - k\cdot \nu &\le c \cdot\gamma - k\cdot \eta = 
\sum_{i \leq n_\Xcal} c_{ij^{(i)}} \mu_i - \sum_{j \leq n_\Ycal} k_j \sum_{i \leq n_\Xcal: j^{(i)}=j} \mu_i \\
&= \sum_{i \leq n_\Xcal} c_{ij^{(i)}} \mu_i - \sum_{i \leq n_\Xcal} k_{j^{(i)}} \mu_i =
\sum_{i \leq n_\Xcal} \left( c_{ij^{(i)}}-k_{j^{(i)}}\right) \mu_i \\
&= \sum_{i \leq n_\Xcal} k^c_i \mu_i = k^c \cdot \mu.
\end{align*} 
This implies
$\OT(\mu, \nu, c) \le k\cdot \nu + k^c\cdot  \mu$.
Since $\OT(\mu, \nu, c) \ge k\cdot \nu + k^c\cdot \mu$ is true by the Kantorovich duality, equality follows, thus $(k^c,k)$ is an optimizer for $\DOT(\mu,\nu, c)$.

To show the converse implication, we now assume that $(k^c,k)$ is an optimizer for $\DOT(\mu,\nu, c)$. This yields $\OT(\mu, \nu, c) -k\cdot \nu = k^c \cdot \mu$. On the other hand, for any $\eta \in \Delta_{n_{\Ycal}}$, we have  $\OT(\mu, \eta, c) \ge k\cdot \eta + k^c\cdot \mu$ by the Kantorovich duality. Hence, we obtain
\[
\OT(\mu, \eta, c) - k\cdot \eta \ge   k^c\cdot \mu =\OT(\mu, \nu, c) - k\cdot \nu,\quad \text{for any $\eta \in \Delta_{n_{\Ycal}}$}.
\] 
Thus $k$ lies in the subdifferential $\partial_\nu \OT(\mu,\nu,c)$.
\end{proof}

\begin{remark}
    Theorem~\ref{thm.subdiffOT} allows us to formulate a more general version of Corollary~\ref{cor:SCNEex}. Namely, we obtain that, when $\mathcal{E}$ is convex, $G(\nu,k) = G(\nu)$ for all $\nu \in \Delta_{n_{\Ycal}}$, $k \in K$, $\nu^\ast \in \text{ri}(\Delta_{n_\Ycal})$, and $K = \mathbb{R}{n_\Ycal}_+$, then a SCNE exists. Indeed, by \cite[Theorem 23.4]{Rockafellar}, an element $k \in -\partial_\nu\OT(\mu,\nu^\ast, c) - \nabla \mathcal{E}(\nu^\ast)$ exists. Since $k + c\cdot 1 \in \partial_\nu\OT(\mu,\nu^\ast, c) - \nabla \mathcal{E}(\nu^\ast)$ for all $c \in \mathbb{R}$, we can find $c\in \mathbb{R}$ such that $k + c \cdot 1 \in \mathbb{R}_+^{n_\Ycal}$. Hence, $(\nu^\ast, k+c \cdot 1)$ is indeed a SCNE by Theorem~\ref{thm:OptimizationPrincipal}.
\end{remark}

\subsection{Approximation Result.}
This section is devoted to the discussion of approximation results for the principal and the agents' problems. We will show that, under some assumptions, up to an error that can be made as small as wanted, both principal and agents can compute (very efficiently) regularized transport problems rather than the original ones; see Remark~\ref{rem:approx}. 

\begin{prop}\label{prop:approx_pr}
Let $K$ be closed and let $\mathcal{E}$ be strictly convex.
Assume that  $G(\nu, k) = G(\nu)$ 
is a continuous function. Let $\nu^\ast$ be a minimizer of $G$.
Moreover, assume that $k^\epsilon \in \left(-\partial_\nu\OT^\epsilon(\mu, \nu^\ast)- \nabla \mathcal{E}[\nu^\ast]\right) \cap K$ with $k^\epsilon (x_0) = 0$ for all $\epsilon>0$. Then there is $k^\ast\in K$ s.t. $\nu^\ast = \text{BR}(k^\ast)$ and
$k^\epsilon \rightarrow k^\ast$. 
Moreover, we have
\[
G(\text{BR}(k^\epsilon)) \rightarrow G(\text{BR}(k^\ast)).
\] 
\end{prop}

\begin{proof}
We first note that, since $\mathcal{E}$ is strictly convex, the map $\text{BR}$ is a function.
By Theorem~\ref{thm:LimitEntropic} and Theorem~\ref{thm.subdiffOT}, $k^\epsilon \rightarrow k^\ast$ with $k^\ast \in -\partial_\nu\OT(\mu, \nu^\ast)- \nabla \mathcal{E}[\nu^\ast]$. Since $K$ is closed, $k^\ast\in K$. Moreover,  $\nu^\ast = \text{BR}(k^\ast)$ by Proposition~\ref{prop:NecessarySufficient}. Finally, by Theorem~\ref{thm:BRContinuous} the map $\text{BR}$ has a closed graph, so the claim follows from continuity of $G$.
\end{proof}

\begin{remark}\label{rem:approx}
When the set $K$ of cost vectors and the principal's cost function $G$ are as in Proposition~\ref{prop:approx_pr}, the principal can look for a cost vector $k^\epsilon$ in $\left(-\partial_\nu\OT^\epsilon(\mu, \nu^\ast)- \nabla \mathcal{E}[\nu^\ast]\right) \cap K$ rather than in $\left(-\partial_\nu\OT(\mu, \nu^\ast)- \nabla \mathcal{E}[\nu^\ast]\right) \cap K$, for $\eps$ small enough so that $G(\text{BR}(k^\epsilon))$ is as close as wanted to the optimal value $G(\text{BR}(k^\ast))$. The reason for doing this is the fact that computing $\partial_\nu\OT^\epsilon(\mu, \nu^\ast)$ is more efficient than computing $\partial_\nu\OT(\mu, \nu^\ast)$. Indeed, $\partial_\nu \OT^\epsilon(\mu,\nu^\ast)$ is the dual optimizer $\psi^\epsilon$ for the regularized problem, that can be approximated via the Sinkhorn algorithm (the scaling variables satisfy $(u,v)=(e^{\varphi^\epsilon/\epsilon)}, e^{\psi^\epsilon/\epsilon})$, see \cite[Propositions 4.4 and 4.6]{ComputationalOT}).

Note also that, by offering $k^\eps$ rather than $k^*$ to the agents, they also achieve a result as close as wanted to their optimal value, for $\eps$ small enough. This follows by Proposition~\ref{thm:PopulationGame} and Theorem~\ref{thm:BRContinuous}, 
since the optimal transport problem is stable w.r.t. its marginals, see \cite{ghosal2022stability}.
Furthermore, again for efficiency reasons, agents can as well decide to solve a regularized OT problem rather than the original one in Proposition~\ref{thm:PopulationGame}, and get as closed as wanted to their optimal value; see \cite{acciaio2021cournot}. \hfill$\diamond$
\end{remark}

For general functions $G(\nu,k)$, i.e. depending on the second variable as well, it might happen that considering only the entropic regularization yields to larger costs than necessary, since it might happen that, in the set of all dual optimizers, better candidates can be found. An illustration of this phenomenon is given in the next example.

\begin{example} 
Consider $\Xcal= \Ycal = \{1,2\}$, $\mu = \tfrac{1}{2} \delta_{\{1\}} + \tfrac{1}{2} \delta_{\{2\}}$, $c= 1_{\{(1,1), (2,2)\}} + 2 \cdot 1_{\{(1,2), (2,1)\}}$, $\mathcal{E} \equiv 0$, $K=\{(1,x):x \in \mathbb{R}\}$ and 
\[
G(\nu, k) = \nu_1^2 + \nu_2^2 + \left(\tfrac{1}{2} + k_2\right )^2,\quad \nu\in\Delta_{n_\Ycal}.
\] 
For any $\nu\in\Delta_{n_\Ycal}$, we denote by  $(\varphi^\ast(\nu),\psi^\ast(\nu))$ the optimizer for $\DOT(\mu, \nu, c)$ that is the unique limit of the dual entropic optimizers.
Note that, by Proposition~\ref{prop:unique}, for all $\nu\in\RR^2$ except $\hat\nu = (\tfrac{1}{2}, \tfrac{1}{2})$ there is a unique dual optimizer, that clearly coincides with $(\varphi^\ast(\nu),\psi^\ast(\nu))$.
This dual optimizer reads as
$\varphi^\ast(\nu)=(0,1)$, $\psi^\ast(\nu)=(1,2)$ if $\nu_1<\tfrac{1}{2}$, and $\varphi^\ast(\nu)=(0,1)$, $\psi^\ast(\nu)=(1,0)$ if $\nu_1>\tfrac{1}{2}$.
The set of all dual optimizers given $\hat\nu$ in $K$ is given by $\varphi = (0,x)$, $\psi = (1,1-x)$ with $x \in [-1,1]$. Hence, by Theorem~\ref{thm:LimitEntropic} we have that $\psi^\ast (\hat\nu) = (1,1)$.
Since $\nu_1^2+\nu_2^2\ge \tfrac{1}{2}$, we now obtain that
\begin{align*}
    G(\nu, -\psi^\ast(\nu)) \ge \begin{cases}
        \tfrac{1}{2} + \frac{9}{4}, &\text{if } \nu_1<\tfrac{1}{2} \\
        \tfrac{1}{2} + \frac{1}{4}, &\text{if } \nu_1=\tfrac{1}{2} \\
        \tfrac{1}{2} + \frac{1}{4}, &\text{if } \nu_1>\tfrac{1}{2}
    \end{cases}.
\end{align*} 
However, choosing $k=(-1, -\tfrac{1}{2})$ yields $G(\hat{\nu},k)=\tfrac{1}{2} < \tfrac{3}{4} \le G(\nu, \psi^\ast(\nu))$ for all $\nu \in\Delta_{n_\Ycal}$. Let us now assume that the principal chooses the vector of costs relying on regularized transport problems, i.e. chooses a pair $(\nu, -\psi^\epsilon(\nu))$. Then, for $\epsilon$ sufficiently small, 
the pair $(\nu,-\psi^\epsilon(\nu))$ is close to $(\nu, -\psi^\ast(\nu))$. Since $G$ is continuous, also the cost $G(\nu,k)$ will be close to the cost $G(\nu, \psi^\ast(\nu)) \ge \frac{3}{4}$. Hence, the cost will be substantially larger than $\tfrac{1}{2}=G(\hat{\nu},k)$. All in all, this shows that relying on regularized transport problems to compute the optimal vector of costs can lead to choices that are not close to the optimal ones.
\hfill $\diamond$
\end{example}

\subsection{Numerics.}\label{sect.num} 
In this section we come back to Example~\ref{ex:SCNE} and illustrate how the presence of a principal affects the choice of the agents. 
Here, we choose $\Xcal=\{1,2\}$, $\Ycal = \{1, \ldots, 10\}$, $f_j(x)=x^2$ for all $j \in \Ycal$, $(g(k))_{k \in \Ycal} = (2,1,0.5,0.25,0.1,0.05,0.02,0,0,0)$ and
\[
c= \begin{pmatrix}
    5&4&3&2&1&1&2&3 \\
    5&1&5&5&1&1&1&5
    \end{pmatrix}.
\]
Since $f_j(\nu_j)$ is differentiable, $g$ is non-negative and $2\sum_{k=1}^{ \lfloor n_\Ycal/2 \rfloor} g(k)< g(0)$, which implies that $(g(|a-j|))_{a,j \in \{1, \ldots, n_\Ycal\}}$ is positive definite, $\mathcal{E}[\nu]$ is strictly convex. Hence, we can compute equilibria for the game with and without a principal by applying the results established in the previous subsections. 

Let us first compute the Cournot-Nash equilibrium of the game without a principal, that is where the costs of the agents are given by \eqref{eq.totcost} with $k=0$. For this we note that, as described in Section~\ref{sec:game_continuum}, it suffices to find the minimizer $\nu^\ast$ of the convex optimization problem
\[
\inf_{\nu \in \mathcal{P}(\Ycal)} \{\OT(\mu,\nu,c) + \mathcal{E}[\nu]\},
\] 
as the CNE is then given by a primal optimizer $\gamma^\ast$ of $\OT(\mu,\nu^\ast,c)$. As established in \cite{acciaio2021cournot}, considering the regularized problem yields good approximations of $\gamma^\ast$. 

For the game with a principal, we note that the minimizer of the cost function $G$ reads  as  $\nu^\ast = (\frac{1}{n_\Ycal}, \ldots, \frac{1}{n_\Ycal})$. Hence, by Corollary~\ref{cor:SCNEex}, a SCNE is given by  $(\gamma^\ast, k^\ast)$, where $\gamma^\ast$ is an optimizer for $\OT (\mu, \nu^\ast,c)$ and $k^\ast \in \left(- \partial_\nu \OT(\mu,\nu^\ast, c) - \nabla \mathcal{E}[\nu^\ast]\right)$. By Proposition~\ref{prop:approx_pr} and Remark~\ref{rem:approx}, we obtain good approximations of $\gamma^\ast$ and $k^\ast$ by computing the optimizers $
\gamma^\eps$ of  $\OT_\eps (\mu, \nu^\ast,c)$ and $ k^\eps \in \left( -\partial_\nu \OT_\eps(\mu,\nu^\ast, c) - \nabla \mathcal{E}[\nu^\ast]\right)$.

We report in Tables \ref{table:CNE} and \ref{table:SCNE}, respectively, the approximate equilibrium strategy $\gamma^\eps$ and the associated approximate equilibrium distribution $\nu^\eps = p_2\#\gamma^\eps$, in the case without a principal (that is taking $k=0$) and with a principal (optimizing over $k$), respectively. Furthermore, for the game with a principal, we report the principal's choice of costs $k^\eps$.

\begin{table}[h]
\centering
	\begin{tabular}{c|cccccccc}
		& 1 & 2 & 3 & 4 & 5 & 6 & 7 & 8   \\
		\hline
		$\gamma^\eps(1,y)$ &  0.0019 & 0.0000 & 0.0030 & 0.0030 & 0.1555 & 0.1336 & 0.0001 & 0.0030 \\
		$\gamma^\eps(2,y)$ &  0.0011 & 0.3115 & 0.0000 & 0.0000 & 0.0868 & 0.0746 & 0.2260 & 0.0000  \\
		\hline
		$\nu^\eps(y)$ &  0.0030 & 0.3115 & 0.0030 & 0.0030 & 0.2423 & 0.2082 & 0.2261 & 0.0030 \\
	\end{tabular}
	\caption{Cournot-Nash equilibrium for the game without a principal}
    \label{table:CNE}
\end{table}

\begin{table}[h]
\centering
	\begin{tabular}{c|cccccccc}
		& 1 & 2 & 3 & 4 & 5 & 6 & 7 & 8  \\
		\hline
		$k^\eps(y)$ &  -2.0357 & 1.8393 & -2.1003 & -1.1610 & 1.7543 & 1.7793 & 1.8393 & -1.9153   \\
		\hline
		$\gamma^\eps(1,y)$ &  0.0000 & 0.0000 & 0.0875 & 0.1250 & 0.0000 & 0.0000 & 0.0000 & 0.0875  \\
		$\gamma^\eps(2,y)$ &  0.1250 & 0.1250 & 0.0375 & 0.0000 & 0.1250 & 0.1250 & 0.1250 & 0.0375  \\
		\hline
		$\nu^\eps(y)$ &   0.1250 & 0.1250 & 0.1250 & 0.1250 & 0.1250 & 0.1250 & 0.1250 & 0.1250  \\
	\end{tabular}
	\caption{Stackelberg-Cournot-Nash equilibrium for the game with a principal}
    \label{table:SCNE}
\end{table}
As expected, the presence of a principal in this game has a clear effect. To wit, in the game without a principal most agents choose one of the actions in $\Ycal_\text{good}:=\{2,5,6,7\}$, which yields $G(\nu^\eps)=0.2502$, whereas in the game with a principal, the equilibrium distribution satisfies $G(\nu^\eps)=0.125$. That the equilibrium distribution in this case is uniform over the actions is clearly due to the principal's choice of $k$, which means that an agent choosing an action from $\Ycal_\text{good}$ is charged an additional cost, whereas an agent choosing a non-preferred action from $\Ycal\setminus \Ycal_\text{good}$ faces reduced costs.

\bibliographystyle{plain}
\bibliography{literatureStackelbergCournotNash}
 
\end{document}